\documentclass[11pt]{article}
\usepackage[margin=1.2in]{geometry}

\usepackage[noprefix]{nomencl}

\makenomenclature

\usepackage{tikz}
\usepackage[T1]{fontenc}

\usepackage{hyperref}

\usepackage[parfill]{parskip}    
\usepackage{latexsym, amssymb, bbm}
\usepackage{amsthm}
\usepackage{xypic}
\usepackage{xcolor}
\usepackage{transparent}
\usepackage{graphicx}
\usepackage{amsmath}

\begingroup
\makeatletter
   \@for\theoremstyle:=definition,remark,plain\do{%
     \expandafter\g@addto@macro\csname th@\theoremstyle\endcsname{%
        \addtolength\thm@preskip\parskip
     }%
   }
\endgroup

\newtheorem{lem}{Lemma}[section]
\newtheorem{thm}[lem]{Theorem}
\newtheorem{prop}[lem]{Proposition}
\newtheorem{cor}[lem]{Corollary}

\newtheorem{remark}[lem]{Remark}

\theoremstyle{definition}
\newtheorem{defn}[lem]{Definition}

\DeclareMathOperator{\supp}{supp}

\DeclareMathOperator{\id}{id}
\DeclareMathOperator{\Int}{int}

\DeclareMathOperator{\Mas}{Mas}
\DeclareMathOperator{\cyl}{cyl}
\DeclareMathOperator{\CZ}{CZ}

\newcommand{\ip}[1]{\langle #1 \rangle}

\newcommand{\abs}[1]{\left| #1 \right|}
\newcommand{\norm}[1]{\left\| #1 \right\|}

\newcommand{\gen}[1]{\left\langle#1\right\rangle}

\newcommand{\ide}{\boldsymbol{\mathbbm{1}}}

\numberwithin{equation}{section}

\def\C{\mathbb{C}}

\def\R{\mathbb{R}}
\def\Z{\mathbb{Z}}

\def\bB{\mathbb{B}}
\def\bD{\mathbb{D}}

\def\cA{\mathcal{A}}
\def\cC{\mathcal{C}}

\def\cH{\mathcal{H}}
\def\cJ{\mathcal{J}}
\def\cL{\mathcal{L}}
\def\cM{\mathcal{M}}
\def\cN{\mathcal{N}}
\def\cO{\mathcal{O}}

\def\a{\alpha}

\def\g{\gamma}
\def\d{\partial}

\def\e{\epsilon}
\def\eps{\epsilon}
\def\i{\iota}

\def\l{\lambda}

\def\th{\theta}

\def\vp{\varphi}

\def\w{\omega}
\def\z{\zeta}
\def\vp{\varphi}

\def\Si{\Sigma}
\def\O{\Omega}

\def\Ham{\mbox{Ham}}

\begin{document}
\title{Bounding Lagrangian widths via geodesic paths}

\author{\textsc Matthew Strom Borman\thanks{Partially supported by NSF grants DMS-1006610 and DMS-1304252.}\, and Mark McLean\thanks{Partially supported by NSF grant DMS-1005365 and a grant to the Institute for Advanced Study by The Fund for Math.}}
\maketitle

\begin{abstract}
The width of a Lagrangian is the largest capacity of a ball that can be symplectically embedded into the ambient 
manifold such that 
the ball intersects the Lagrangian exactly along the real part of the ball.  Due to Dimitroglou Rizell, finite width is an obstruction to a Lagrangian admitting an exact Lagrangian cap in the sense of Eliashberg--Murphy.
In this paper we introduce a new method for bounding the width of a Lagrangian $Q$
by considering the Lagrangian Floer cohomology of an auxiliary Lagrangian 
$L$ with respect to a Hamiltonian whose chords correspond to geodesic paths in $Q$.
This is formalized as a wrapped version of the Floer--Hofer--Wysocki capacity 
and we establish an associated energy-capacity inequality with the help of a closed-open map.
For any orientable Lagrangian $Q$ admitting a metric of non-positive sectional curvature 
in a Liouville manifold, we show the width of $Q$ is bounded above by four times its displacement energy.
\end{abstract}

\setcounter{tocdepth}{2}
\tableofcontents


\section{Introduction}

\subsection{The width of a Lagrangian}

Given $Q^{n} \subset (M^{2n}, \w)$ a closed Lagrangian submanifold in a symplectic manifold we will
consider the following relative symplectic embedding problem first considered by Barraud--Cornea
\cite{BC06, BC07}.  If $\bB^{2n}_{R} = \{z \in \C^{n} : \pi\abs{z}^{2} \leq R\}$ denotes the ball of capacity $R$ in the standard 
$(\C^{n}, \w_{0} = dx \wedge dy)$, define a symplectic embedding 
$\iota: (\bB^{2n}_{R}, \w_{0}) \to (M^{2n}, \w)$ to be \textbf{relative} to $Q$ if 
$$
\iota^{-1}(Q) = \bB^{2n}_{R} \cap \R^{n}
$$ 
and define the \textbf{width} of the Lagrangian $Q$ to be
$$
	w(Q; M) := \sup \{R : \bB^{2n}_{R} \mbox{ embeds symplectically in $(M, \w)$ relative to $Q$}\}.
$$
Recall that a compact subset $X \subset (M, \w)$ of a symplectic manifold is \textbf{displaceable} if there is a 
compactly supported Hamiltonian diffeomorphism $\vp \in \Ham_{c}(M, \w)$ such that $\vp(X) \cap X = \emptyset$.
The displacement energy $e(X; M)$ is the least Hofer energy needed to displace $X$, the precise definition appears in 
Section~\ref{s:IntroECI}.

Previous methods for bounding Lagrangian widths, which we review in
Section~\ref{s:previous}, assumed $Q$ was monotone and used Lagrangian Floer cohomology $HF^{*}(Q)$ to prove
$Q$ was uniruled by holomorphic disks \cite{Al05, BC09, Ch10}.
In the non-monotone case, more refined methods have been suggested by Cornea--Lalonde \cite{CL05, CL06} and 
Fukaya \cite{Fu06}, though the analytic foundations of each approach remain to be completed.

\subsubsection{An overview of our method}\label{s:brief}

The focus of this paper will be the introduction of a new technique for bounding the width of a Lagrangian $Q$. 
In contrast to previous work we will not need to assume $Q$ is monotone and we will not use the sophisticated machinery 
behind other suggested approaches.  The main idea will be to consider Lagrangian Floer cohomology $HF^{*}(L;H_{Q})$,
generated by Hamiltonian chords for an auxiliary exact Lagrangian $L \subset (M, d\th)$ in a Liouville manifold, 
where the Hamiltonian $H_{Q}$
induces geodesic flow in a Weinstein neighborhood of $Q$.  With this set-up we are able to use 
Lagrangian Floer theory in its simplest form, the case of an exact Lagrangian, and the Hamiltonian $H_{Q}$ takes the role of
$Q$.  At a functional level we have replaced $CF^{*}(L, Q)$, with all of the potential complications inherent in Lagrangian Floer theory
in the general case, with the well-behaved object $HF^{*}(L;H_{Q})$.  

Given a relatively embedded ball $\bB^{2n}_{R}$ for $Q$, we will pick an auxiliary exact Lagrangian $L$ that agrees
with the imaginary axis in the ball and can be displaced from $Q$ by a Hamiltonian isotopy.  By construction
the center of the ball $q_{0} \in Q \cap L$ is a generator in the chain complex $CF^{*}(L;H_{Q})$ and the fact
$L$ can be displaced from $Q$ leads to the existence of a differential connecting a chord $x$ and the constant chord $q_{0}$.
In favorable cases, in particular if $x$ does not represent a geodesic loop based at $q_{0}$, from such differentials
we are able to extract a holomorphic curve in $\bB^{2n}_{R}$ whose energy gives bounds on the size of the ball.

The procedure of looking for chain level information in $CF^{*}(L; H)$, where $H$ is adapted to a compact subset $X \subset M$
is formalized as a wrapped version $c^{FHW}_{L}(X)$ of the Floer--Hofer--Wysocki capacity.
Via a closed-open map between Hamiltonian and Lagrangian Floer cohomology, we relate the wrapped version
to the standard Floer--Hofer--Wysocki capacity.  This leads to the energy-capacity inequality, which bounds
$c_{L}^{FHW}(X)$ by the displacement energy $e(X; M)$.  
Going back to the special case of $CF^{*}(L;H_{Q})$, the capacity
$c_{L}^{FHW}(Q)$ gives bounds on the energy of the holomorphic curve we construct, and hence by the energy-capacity inequality
we have a bound
on the size of the ball in terms of the displacement energy $e(Q; M)$.

\subsubsection{Bounds on Lagrangian widths}

With our method we get the following bound on the width of a displaceable Lagrangian $Q \subset (M, \w)$
in terms of the displacement energy $e(Q; M)$.

\begin{thm}\label{t:main}
	If $(M, \w)$ is a Liouville manifold and $Q \subset M$
	is a closed oriented Lagrangian that is displaceable, then
	\begin{equation}\label{e:mainbound}
		w(Q; M) \leq 4\, e(Q; M)
	\end{equation}
	provided $Q$ admits a Riemannian metric with non-positive sectional curvature.
\end{thm}

Since $(\C^{n}, \w_{0})$ is a Liouville manifold, see Section~\ref{s:Preliminaries} for the definition,
and any compact set in $\C^{n}$ is displaceable, by Theorem~\ref{t:main} the width of any
closed oriented Lagrangian in $\C^{n}$ is finite if it admits a metric of non-positive curvature.
The most basic examples are Lagrangian tori and in $\C^{2}$ these are the only orientable Lagrangians.
Recall by the Gromov--Lees theorem 
\cite{Gr71, Le76, ALP94} that $S^{1} \times L^{n-1}$ can be embedded as a Lagrangian in $\C^{n}$ whenever $TL \otimes \C$ is trivial.   Therefore other examples, for instance, are given by any Lagrangian of the form $Q = S^{1} \times L$ where $L$ is a closed orientable manifold admitting a metric of non-positive curvature with dimension at most $3$.  
Unfortunately we know of no example where the inequality \eqref{e:mainbound} is sharp.

In the context of the overview given above, the non-positive curvature assumption in Theorem~\ref{t:main} gives restrictions on the type of chords that can be connected to the center of the ball via a differential and serves to ensure we can extract a non-trivial holomorphic curve in the ball.  It is likely this assumption can be weakened to the existence of a metric on $Q$ such that
\begin{equation}\label{e:newAssumption}
	\mbox{$m_{\O}(q) \not= 1 - \mu_{Q}(v)$}
\end{equation}
for all based geodesic loops $q: [0,1] \to Q$ and maps $v \in \pi_{2}(M, Q)$.
Here $m_{\O}(q)$ is the Morse index of the geodesic and $\mu_{Q}(v)$ is the Maslov index.  
It is only the proof of Lemma~\ref{l:lowerorderchord} that kept us from using this weaker assumption.

Note that the assumption \eqref{e:newAssumption} is weaker than $g$ having non-positive section curvature,
since $\mu_{Q} \in 2\Z$ whenever $Q$ is orientable and non-positive curvature implies $m_{\O} \equiv 0$.	
With more work it is conceivable that this 
new assumption \eqref{e:newAssumption} could be weakened further by requiring $v$ to be a holomorphic
disk such that $v(\d\bD)$ is homotopic in $Q$ to $q$.  This will require controlling the limits of differentials
connecting a constant chord $q_{0} \in Q \cap L$ to a chord representing a geodesic loop based $q_{0}$. 

\subsubsection{Finite width as an obstruction to flexibility}

A closed Lagrangian $Q \subset (M, d\th)$ is said to \textbf{admit an exact Lagrangian cap} if there is a Liouville subdomain
$(W, d\th|_{W}) \subset (M, d\th)$ such that $Q \backslash \Int W$ is a non-empty exact Lagrangian and
$\th|_{Q \backslash \Int W}$ admits a primitive vanishing on its boundary $Q \cap \d W$. 
In \cite{Di13} Dimitroglou Rizell made the following fantastic observation.
\begin{thm}[\cite{Di13}]
	If a closed Lagrangian $Q \subset (M, d\th)$ admits an exact Lagrangian cap, then $Q \subset M$ has infinite width.
\end{thm}
For the case of $M = \C^{n}$ and $W = \bB^{2n}$ such Lagrangians were built by Ekholm--Eliashberg--Murphy--Smith
\cite{EEMS13} when $n \geq 3$.
The construction of these Lagrangians used Murphy's \cite{Mu12} h-principle for loose Legendrians
and its extension to an h-principle for exact Lagrangian caps in 
$(\C^{n}\backslash \bB^{2n}, \d\bB^{2n})$ by Eliashberg--Murphy \cite{EM13}. 

Let us point out that the known examples of oriented Lagrangians in $\C^{n}$ with infinite width seem to fit with our method's potential
extension to the requirement in \eqref{e:newAssumption}.  For example, \cite{EEMS13} built 
Lagrangian $S^{1} \times S^{2k} \subset \C^{2k+1}$ that admit an exact Lagrangian cap and
have a holomorphic disk with Maslov index $2-2k$.  Condition \eqref{e:newAssumption} fails for these Lagrangians since
they must have a based geodesic with Morse index $2k-1$.

Since finite width is an obstruction to admitting an exact Lagrangian cap, we have the following corollary of Theorem~\ref{t:main}
\begin{cor}\label{c:NoCaps}
	A closed orientable displaceable Lagrangian $Q \subset (M, d\th)$ does not admit an exact Lagrangian cap if the manifold $Q$
	admits a metric with non-positive sectional curvature.
\end{cor}
In \cite[Theorem 1.7]{Ch12}, Chantraine gave an example of a Lagrangian torus $T \subset \C^{2}$
such that $T \backslash \Int \bB^{4}_{R}$ is never an 
exact Lagrangian cap, though this terminology was not used. Corollary~\ref{c:NoCaps} shows that this is a much more 
general phenomenon.

\subsubsection{Previous width bounds via uniruling by $J$-holomorphic curves}\label{s:previous}

Given an almost complex structure $J$ on $(M, \w)$ a $J$-holomorphic curve is a smooth
function $u: (S, j) \to (M, J)$ from a Riemann surface $S$ to $M$ that satisfies 
the Cauchy-Riemann equation $du \circ j = J \circ du$.  Building on Gromov's \cite{Gr85} proof
of absolute packing obstructions, so far all non-trivial upper bounds for the width of a Lagrangian have 
gone through uniruling results for the Lagrangian by holomorphic curves via the following lemma.

\begin{lem}\label{l:uniruling}
Let $Q \subset (M, \w)$ be a closed Lagrangian. Suppose there is a constant $A \geq 0$
so that for all points $q \in Q$ and compatible almost complex structures $J$ on $(M, \w)$,
there is a non-constant $J$-holomorphic curve $u:(\Si, \d\Si) \to (M, Q)$ with $q \in u(\d\Si)$ and 
$\int_{\Si} u^{*}\w  \leq A$.  Then $w(Q; M) \leq 2A$.
\end{lem}
The proof, see for instance \cite[Corollary 3.10]{BC07}, 
goes by taking a symplectic embedding $\iota: \bB^{2n}_{R} \to (M, \w)$ relative to $Q$ and picking 
$q = \iota(0)$ and a compatible almost complex structure so that $\iota^{*}J = J_{0}$ is the standard complex structure
on $\C^{n}$.  After applying Schwarz reflection across $\R^{n}$ to the part of the $J$-holomorphic curve $u$ that is in the ball 
$\bB^{2n}_{R}$, the standard monotonicity estimate gives $R \leq 2\int_{\Si} u^{*}\w$ and hence $w(Q; M) \leq 2A$.

To date the main approaches to proving such uniruling results have involved using a flavor of Lagrangian Floer cohomology for $Q$ and hence work best when $Q$ is monotone, i.e.\ 
symplectic area $\w: \pi_{2}(M, Q) \to \R$ and the Maslov index $\mu_{Q}: \pi_{2}(M, Q) \to \Z$
are proportional $\w = \l\, \mu_{Q}$ for some $\l \geq 0$.
When $Q$ is displaceable, then using 
$HF^{*}(Q) = 0$ and action considerations one gets uniruling results for $Q$ with disks of area 
at most $e(Q; M)$.  Hence for displaceable monotone Lagrangians one has
\begin{equation}\label{e:conj}
	w(Q; M) \leq 2\,e(Q; M)\,.
\end{equation}
This is the route taken by Albers \cite{Al05} and Biran--Cornea \cite{BC09}, see also \cite{Da12, EK11}.
In \cite{Ch10, Ch14} Charette proved a stronger form of uniruling in the monotone case, which was conjectured
by Barraud--Cornea \cite[Conjecture 3.15]{BC06}.
When $Q$ is non-displaceable Biran--Cornea \cite{BC09} used the ring structure of $HF^{*}(Q)$ to detect uniruling for holomorphic disks of a given Maslov index, which due to monotonicity give area bounds.

\subsection{The Floer--Hofer--Wysocki capacity and its relative version}

In \cite{FHW94} Floer--Hofer--Wysocki introduced a capacity for open subsets of $\C^{n}$ using
a symplectic homology \cite{FH94} construction.  For sample applications see \cite{FHW94, CFHW96, He00, He04, Dr08, Ir12}.
In this paper we will utilize a modified version of the Floer--Hofer--Wysocki capacity
and we will introduce the analogous Lagrangian version, which is defined via a wrapped Floer cohomology construction 
\cite{AbS10}.  These capacities are related via a closed-open map between Hamiltonian Floer cohomology and Lagrangian Floer cohomology and have energy-capacity inequalities, which are established in Theorem~\ref{t:CapIneq}.

\subsubsection{The definitions of the capacities}

Here we will give a brief descriptions of the Floer--Hofer--Wysocki capacities,
see Sections~\ref{ss:LFC} and \ref{s:DefFHC} for a more thorough description as well as our Floer theory conventions.

Given a Liouville manifold $(M^{2n}, d\th)$ and a compact subset $X \subset M$, consider the set of Hamiltonians
$$
	\cH^{X} = \{ H \in C^{\infty}(S^{1}\times M) : H|_{S^{1} \times X} < 0 \mbox{ and supp$(dH)$ is compact}\}.
$$
Using filtered Hamiltonian Floer cohomology, for $a > 0$ one sets
$$
	HF^{*}(X, a):= \varinjlim_{H \in \cH^{X}} HF^{*}_{(-a, 0]}(H)
$$
where the direct limit is given by monotone continuation maps 
$HF^{*}_{(-a, 0]}(H_{0}) \to HF^{*}_{(-a, 0]}(H_{1})$
that exist whenever $H_{0} \leq H_{1}$.  
The \textbf{Floer--Hofer--Wysocki capacity} of $X$ is
$$
	c^{FHW}(X) = \inf\{ a > 0 : i_{X}^{a}(\ide_{M}) = 0\}
$$
where $i_{X}^{a}: H^{*}(M) \to HF^{*}(X, a)$ is a natural map described in Section~\ref{s:DefFHC}.
If $i_{X}^{a}(\ide_{M})$ is never zero, then $c^{FHW}(X) := + \infty$

Suppose one has an exact Lagrangian $L \subset M$,
then one can repeat the construction of the Floer--Hofer--Wysocki capacity in the 
filtered Lagrangian Floer cohomology setting.
In particular for $a > 0$ one sets
$$
	HF^{*}(L; X, a):= \varinjlim_{H \in \cH^{X}} HF^{*}_{(-a, 0]}(L; H)
$$
and defines the \textbf{Lagrangian Floer--Hofer--Wysocki capacity} (relative to $L$) of $X$ as
$$
	c^{FHW}_{L}(X) = \inf\{ a > 0 : i_{L;X}^{a}(\ide_{L}) = 0\}
$$
where $i_{L;X}^{a}: H^{*}(L) \to HF^{*}(L;X, a)$ is a natural map described in Section~\ref{s:LAS}.

\subsubsection{The comparison and energy-capacity inequalities}\label{s:IntroECI}

Recall that the \textbf{displacement energy} $e(X;M)$ of a closed set $X \subset (M, \w)$ is
\begin{equation}\label{e:dis}
	e(X;M) = \inf\{ \norm{H} : H \in C^{\infty}_{c}(S^{1} \times M) \mbox{ and }
	\vp_{H}^{1}(X) \cap X = \emptyset\}
\end{equation}
where $\norm{H} = \max_{t \in S^{1}}(\max_{M} H_{t} - \min_{M} H_{t})$ and
$\vp_{H}^{1} \in \Ham_{c}(M,\w)$ is the time-one map generated by 
the time-dependent Hamiltonian vector field $X_{H_{t}}$ given by
$$
	-dH_{t} = \w(X_{H_{t}}, \cdot) \quad\mbox{where $H_{t}(m) = H(t,m)$ for $t \in S^{1} = \R/\Z$}.
$$
See \cite[Lemma 5.1.C]{Po01} for the proof that this definition of displacement energy is equivalent to the one using Hofer's metric 
\cite{Ho90}.
The \textbf{relative displacement energy} $e_{L}(X; M)$ is 
\begin{equation}\label{e:disL}
	e_{L}(X; M) := \inf\{ \norm{H}_{L} : H \in C^{\infty}_{c}(S^{1} \times M) 
	\mbox{ and } \vp_{H}^{1}(L) \cap X = \emptyset\}
\end{equation}
where $\norm{H}_{L} = \max_{t \in S^{1}} (\max_{L} H_{t} - \min_{L} H_{t})$. 

The Floer--Hofer--Wysocki capacities have the following inequalities, which we prove in Section~\ref{s:EAILP}.
See Definition~\ref{d:admissibleLag} for the definition of an admissible Lagrangian.
\begin{thm}\label{t:CapIneq}
	Let $X \subset M$ be a compact set and $L \subset M$ be an admissible Lagrangian.
	\begin{enumerate}
	\item[(i)] The Hamiltonian capacity bounds the Lagrangian capacity:
	 $c^{FHW}_{L}(X) \leq c^{FHW}(X).$
	\item[(ii)] Energy-capacity inequalities:
	$c^{FHW}_{L}(X) \leq e_{L}(X; M)$ and $c^{FHW}(X) \leq e(X; M).$
	\end{enumerate}
\end{thm}

The proof of part (i) is an immediate consequence of the existence of a closed-open map
$$
	\cC\cO: HF^{*}(X,a) \to HF^{*}(L;X,a)
$$
that is compatible with the maps $i_{X}^{a}$ and $i_{L;X}^{a}$.
The proof of part (ii) is based on an observation by Ginzburg from \cite{Gi10} and the now 
standard argument relating action and displacement energy.
 
While we do not use the relative Lagrangian energy-capacity inequality $c^{FHW}_{L}(X) \leq e_{L}(X; M)$ in this
paper, Humili\`{e}re--Leclercq--Seyfaddini \cite{HLS13} have recently
used such an inequality to prove $C^{0}$-rigidity for coisotropic submanifolds.
They use a version of the Hofer--Zehnder capacity developed by Lisi--Rieser \cite{LR13}.

To prove Theorem~\ref{t:main} we will only use the following immediate corollary of Theorem~\ref{t:CapIneq}:
\begin{cor}\label{c:Main}
	If $X \subset M$ is compact and $L \subset M$ is an admissible Lagrangian,
	then
	$$
		c^{FHW}_{L}(X) \leq e(X; M).
	$$
\end{cor}
It would be interesting to see a direct proof of Corollary~\ref{c:Main} that does not go through the Hamiltonian 
Floer--Hofer--Wysocki capacity $c^{FHW}$.

\subsection{An overview of the proof of the main result}\label{s:overviewProof}

Since the proof of Theorem~\ref{t:main} comprises the bulk of the paper, we will now give an in-depth overview 
of how we set up the argument and use Corollary~\ref{c:Main} to extract the needed holomorphic curve
used to establish the bound in Theorem~\ref{t:main}.

\subsubsection{An auxiliary Lagrangian and the Hamiltonian $H_{Q}$} \label{ss:auxlag}
Let us fix a symplectic embedding relative to the Lagrangian $Q$
$$\i: \bB^{2n}_{R} \to (M^{2n}, d\th)$$
and in Lemma~\ref{l:L} we will introduce an auxiliary Lagrangian $L \subset (M, d\th)$ such that:
\begin{itemize}
\item[(i)] $L$ is exact, diffeomorphic to $\R^{n}$, is properly embedded in $M$, and displaceable
from $Q$.
\item[(ii)] In a small Weinstein neighborhood $\cN$ of $Q$ the Lagrangian $L$ is modeled on cotangent fibers
$T^{*}_{q}Q$ for the points $q \in Q \cap L$.
\item[(iii)] $L$ intersects the ball only along the imaginary axis, i.e.\ $\i^{-1}(L) = i\R^{n} \cap \bB^{2n}_{R}$.
\end{itemize}

We will study the Lagrangian Floer--Hofer--Wysocki capacity $c^{FHW}_{L}(Q)$ using the following 
class of functions in $\cH^{Q}$.
Given a metric $g$ on $Q$ we will take Hamiltonians $H_{Q}: M \to \R$ in $\cH^{Q}$
such that $dH_{Q}$ is supported in a small Weinstein neighborhood $\cN$ of $Q$ and 
$$
\mbox{$H_{Q}(q,p) = f_{H_{Q}}(\abs{p}_{g})$}
$$ 
in cotangent bundle $T^{*}Q$ coordinates $(q,p)$ in $\cN$.
See Figure~\ref{f:cHcN} and Section~\ref{s:cHcN} for a precise description of the Hamiltonians we use.
For our choice of $H_{Q}$ and $L$ the non-constant Hamiltonian chords of $L$
correspond to geodesic paths in $Q$ starting and ending at points in $Q \cap L$. 

\begin{remark}
As the proof of Lemma~\ref{l:L} will show, a Lagrangian $L$ with properties (i) and (ii) can be built if 
$(M, \w)$ is the completion of a compact symplectic manifold with a contact type convex boundary.  For property (iii) we use the global Liouville flow on $(M, d\th)$ and this is the main point in the paper where we use the global Liouville structure in an essential way.
\end{remark}

\subsubsection{Using the energy-capacity inequality}

The Lagrangian Floer--Hofer--Wysocki capacity $c^{FHW}_{L}(Q)$ is defined using
Lagrangian Floer cohomology $HF^{*}(L; H_{Q})$, which is generated on the chain level by Hamiltonian chords 
$$
x: [0, 1] \to M \quad\mbox{with $\dot{x}(t) = X_{H_{Q}}(x(t))$ and $x(0), x(1) \in L$}.
$$
Now by Corollary~\ref{c:Main} we have 
$$
c^{FHW}_{L}(Q) \leq e(Q; M)
$$
and we have set up the Lagrangian capacity so that this implies the following:
There is a Hamiltonian chord $x$ for $H_{Q}$ corresponding to a geodesic in $Q$ and a solution to the Floer equation
\begin{equation}\label{e:FEI}
	\begin{cases}
	u = u(s,t): \R \times [0,1] \to M\\
	\d_{s}u + J(u)(\d_{t}u - X_{H_{Q}}(u)) = 0\\
	u(\R \times \{0,1\}) \subset L
	\end{cases}
\end{equation}
with bounded energy
$$E(u) := \int \norm{\d_{s}u}^{2}_{J} ds\,dt \leq e(Q; M)$$
so that $u(-\infty, t) = q_{0}$ and $u(+\infty, t) = x(t)$.
Here $q_{0} = \i(0) \in Q \cap L$, the center of the ball, is
 a constant chord since $dH_{Q} = 0$ on $Q$.
 
\begin{figure}[h]
  \centering
  \def\svgwidth{400pt}
  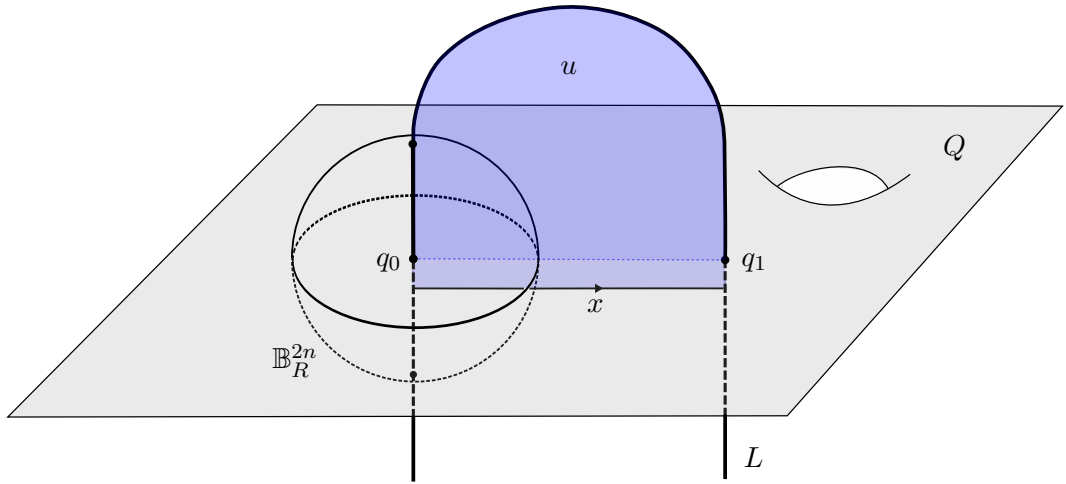
  \caption{In blue a differential $u: \R \times [0,1] \to M$ from \eqref{e:FEI} with boundary on $L$ that connects the chord $x$ to the constant chord $q_{0}$.  The differential $u$ is $J$-holomorphic away from $Q$.}
  \label{f:setup}
\end{figure}

\subsubsection{Building a holomorphic curve from a limit of Floer differentials}\label{ss:CompIdea}

In Section~\ref{ss:limit} we use such solutions \eqref{e:FEI} to the Floer equation,
which are depicted in Figure~\ref{f:setup}, to prove Theorem~\ref{t:main}
in the following way.  Since the Hamiltonian vector field vanishes $X_{H_{Q}} = 0$ outside of the Weinstein neighborhood $\cN$, 
it follows the part of the solution $u$ from \eqref{e:FEI} that maps to $M \backslash \cN$
\begin{equation}\label{e:HCS}
	w = u|_{u^{-1}(M \backslash \cN)} : S \to M\backslash\cN
\end{equation}
is $J$-holomorphic with boundary on $L$ and $\d\cN$.  By shrinking
the fiber diameter of the Weinstein neighborhood $\cN$ to zero and taking a limit, 
via Fish's \cite{Fi11} compactness result we can extract a $J$-holomorphic map
$w_{\infty}: S \to M$ with boundary on $L$ and $Q$ and with energy $E(w_{\infty}) \leq e(Q; M)$.
    
Since we removed the part of the Floer solution in
the Weinstein neighborhood, which contains the center of the ball $q_{0}$, a priori we cannot ensure
that the image of $w_{\infty}$ enters the ball.  
However if the chord $x$ does not correspond to a geodesic loop in $Q$ based at $q_{0}$, then due to the boundary conditions in \eqref{e:FEI} for topological reasons the holomorphic curve $w$ in 
\eqref{e:HCS} must still have part of its boundary on $L$ pass through the ball near $q_{0}$.
By taking the almost complex structure to be standard $J = \i_{*}J_{0}$ in the image of the ball and using Schwarz reflection across $\R^{n}$ and $i\R^{n}$ as needed we can extract from $w_{\infty}$ a non-constant holomorphic curve
$v: S \to (\bB_{R}^{2n}, J_{0})$ passing through $0$ with energy $E(v) \leq 4\,e(Q;M)$.
It then follows from the monotonicity estimate that $R \leq 4\,e(Q;M)$
and hence Theorem~\ref{t:main} is proved.

\subsubsection{Ruling out chords that represent based geodesic loops}

It remains to prove Theorem~\ref{t:ExistsPCD}, which asserts we can assume the chord $x$ does not correspond to a geodesic starting and ending at the same point $q$, and this proof is carried out in Section~\ref{s:pathchordexistence}.  A key element of this proof
is that the Liouville class $\th|_{Q}$ gives an additional filtration on the chain complex $CF^{*}_{(-\infty, 0]}(L;H_{Q})$ and detects when differentials leave the Weinstein neighborhood $\cN$, which is established in Section~\ref{s:nuf}.  Then in Section~\ref{s:CBAB} we use the Liouville-filtration to obtain a bound on a quantity we call the cotangent bundle action $\cA_{H_{Q},L}^{T^{*}Q}(x)$ of the chord $x$.

As we spell out in Section~\ref{s:cHcN} besides constant chords,
there are chords where $f_{H_{Q}}'' > 0$, called near chords, and chords where $f_{H_{Q}}'' < 0$, called far chords.  If $Q$ has a metric of non-positive curvature, in Section~\ref{s:nearchord} we use the Liouville-filtration and the cotangent bundle action bound to show we can assume $x$ is a near chord.  Finally we prove an index relation Proposition~\ref{p:indexR} in Section~\ref{s:indexnear}, which implies that if $x$ is a near chord then it does not correspond to a geodesic starting and ending at the same point in $Q$.


\subsection{Further discussion}

\subsubsection{Studying $Q$ via the Hamiltonian $H_{Q}$}

The idea of proving things about a Lagrangian $Q$ using a Hamiltonian $H_{Q}$ that induces geodesic flow in 
a Weinstein neighborhood of $Q$ goes back to at least Viterbo's \cite{Vi90a, Vi90b} proof of Maslov class 
rigidity for Lagrangian tori in $\C^{n}$.  It was also in \cite{Vi90a} that the relationship between the 
Conley--Zehnder index of a Hamiltonian orbit, the Morse index of the underlying geodesic, and the 
Maslov index was established, and Proposition~\ref{p:indexR} represents the analogous relation for 
Hamiltonian chords. Kerman and {\c{S}}irik{\c{c}}i \cite{Ke09, KS10} later developed methods for proving 
such Maslov class rigidity results using a `pinned' action selector in Hamiltonian Floer theory for this 
type of Hamiltonian and this approach is summarized nicely in \cite[Section 3.2]{Gi11}.  
The Floer--Hofer--Wysocki capacities can be seen as the capacities associated to such `pinned' action selectors.

The limiting argument we use to extract a holomorphic curve from a sequence of Floer differentials for $CF^{*}(L; H_{Q})$
was inspired by the analogous argument in the Hamiltonian Floer context, which appeared in Viterbo and Hermann's papers \cite{Vi99, He00, He04}. By considering $CF^{*}(H_{Q})$ where $dH_{Q}$ is supported in $\cN$ and shrinking $\cN$ to $Q$, from
differentials they extract a holomorphic curve with boundary on the Lagrangian $Q$.  However it is not clear in this setting how to
ensure the differentials pass through $\bB^{2n}_{R} \backslash \cN$, which is needed to be able to conclude
the resulting holomorphic curve passes through the ball.  We overcome this issue by using Lagrangian Floer cohomology for the auxiliary Lagrangian $L$, since we can force our differentials to pass through $\bB^{2n}_{R} \backslash \cN$ in a topologically non-trivial way and hence survive the limiting process.

\subsubsection{Replacing the Lagrangian $Q$ by the Hamiltonian $H_{Q}$}

As remarked at the beginning of the paper, one way of thinking about our method
is that $CF^{*}(L; H_{Q})$ is used as a proxy for the Lagrangian Floer complex $CF^{*}(L, Q)$ 
where $H_{Q}$ replaces the Lagrangian $Q$.  As illustrated in Figure~\ref{f:setup}, given Lagrangians $Q, L \subset (M, \w)$ in a symplectic manifold and distinct points $q_{0}, q_{1} \in Q \cap L$, there is a strong similarity between
holomorphic strips defining the differential for $CF^{*}(L,Q)$
\begin{equation}\label{e:PLFMS}
\begin{cases}
	v= v(s,t): \R \times [0,1] \to M\\
	\d_{s}v + J(u)\d_{t}v = 0\\
	v(\R \times 0) \subset L \quad\mbox{and}\quad v(\R \times 1) \subset Q\\
	v(-\infty, t) = q_{0} \quad\mbox{and}\quad v(+\infty, t) = q_{1}
\end{cases}
\end{equation}
and solutions to the Floer equation defining the differential for $CF^{*}(L; H_{Q})$
\begin{equation}\label{e:CLFMS}
	\begin{cases}
	u = u(s,t): \R \times [0,1] \to M\\
	\d_{s}u + J(u)(\d_{t}u - X_{H_{Q}}(u)) = 0\\
	u(\R \times \{0,1\}) \subset L\\
	u(-\infty, t) = q_{0} \quad\mbox{and}\quad u(+\infty, t) = x(t)
	\end{cases}
\end{equation}
where the chord $x$ represents a geodesic path in $Q$ from $q_{0}$ to $q_{1}$.
In particular the correspondence between $CF^{*}(L; H_{Q})$ and $CF^{*}(L, \vp_{H_{Q}}(L))$ should let
one construct solutions to \eqref{e:PLFMS} by taking limits of solutions to \eqref{e:CLFMS}.

It was suggested by Biran that it would be nice to turn this similarity into a precise relationship between the chain complexes 
$CF^{*}(L, Q)$ and $CF^{*}(L; H_{Q})$.  For certain applications, in situations where $L$ is monotone and $Q$ is not, such a relationship could let one use Lagrangian Floer theory in the monotone setting $HF^{*}(L; H_{Q})$ to stand in for the perhaps undefined 
$HF^{*}(L, Q)$.

\subsubsection{$HF^{*}(L; H_{Q})$ as a deformation of wrapped Floer cohomology}

In this paper we take a very hands-on approach to working with $HF^{*}(L;H_{Q})$.  However let us step back for a moment
to give a different conceptual way of thinking about our argument and its relation to other work.

From \cite{Vi99, SW06, AS06, AS10, Ab12} we know symplectic cohomology
and wrapped Floer cohomology recover the Morse homology of the free and based loop spaces (over $\Z/2$)
$$
	SH^{*}(T^{*}Q) \cong H_{-*}(\Lambda Q) \quad\mbox{and}\quad HW^{*}(T^{*}_{q}Q) \cong H_{-*}(\O_{q}Q).
$$
Moreover given exact Lagrangians $Q, L \subset (M, d\th)$ intersecting transversely $Q \cap L = \{q_{i}\}_{i=0}^{k}$ where $Q$ is closed
and $L$ is open, properly embedded, and $\th|_{L} = 0$, then
there are Viterbo restriction maps \cite{Vi99, AbS10}
$$
	SH^{*}(M) \to SH^{*}(T^{*}Q) \quad\mbox{and}\quad
	HW^{*}(L; M) \to HW^{*}\!\left(\bigcup_{i} T^{*}_{q_{i}}Q\right).
$$
Note that at the chain level $CW^{*}\!\left(\bigcup_{i} T^{*}_{q_{i}}Q\right)$ is generated by
geodesic paths in $Q$ with endpoints in $Q \cap L$ including the constant geodesics $q_{i} \in Q \cap L$.

When $Q$ is not exact there are not such restriction maps (as written), since for example 
there cannot be a ring map from $SH^{*}(\C^{n}) = 0$ to $SH^{*}(T^{*}Q) \not=0$.  Going through
the construction of the Viterbo restriction map one sees that it is necessary to deform $SH^{*}(T^{*}Q)$
to account for differentials that leave a Weinstein neighborhood of $Q \subset M$ and connect orbits for a Hamiltonian of the form 
$H_{Q}$.  The story is similar in our case, where we locate and use differentials in $CF^{*}(L; H_{Q})$, as in Figure~\ref{f:setup}, that
do not arise in $CW^{*}\!\left(\bigcup_{i} T^{*}_{q_{i}}Q\right)$.

Building on Fukaya's \cite{Fu06} wonderful idea of linking the compactified moduli space of 
holomorphic disks on $Q$ with the string topology operations on $\Lambda Q$, 
Cieliebak--Latschev \cite{CL09} have a program to bring such ideas into symplectic field theory.  In particular they have ongoing work
to build a twisted version of Viterbo's map $SH^{*}(M) \to SH^{*}_{tw}(T^{*}Q)$ in terms of a Maurer--Cartan element of $SH^{*}(T^{*}Q)$
when $Q$ is not exact.
As this paper shows the deformation in wrapped Floer cohomology also has applications and it would be
interesting to determine its underlying algebraic nature.

\subsubsection{Other ball packing questions}
The width of a Lagrangian is the relative version of the Gromov width \cite{Gr85} and more generally represents a relative version of the symplectic packing problem, which in its prototypical form is the study of obstructions, beyond volume, 
to symplectic embeddings \cite{Tr95, Bi99, Sc05a, MS12, HK13, Hu10, BH11}. 
Via the symplectic blow-up, the symplectic packing problem is connected with algebraic geometry 
as established by work of McDuff--Polterovich \cite{MP94} and Biran \cite{Bi01}.
This connection was extended to relative packings by Rieser \cite{Ri10}.  Let us also mention
that obstructions to symplectic packings arise in work of Fefferman--Phong \cite{FP82} in connection with the uncertainty principle.

In this paper we solely focus on studying the obstruction to symplectically embedding a single ball $\bB^{2n}_{R}$ into $(M, \w)$ relative to a Lagrangian $Q$.  See \cite{Bu10, Sc05, Ri10} for constructions of relative embeddings.
One can also study the embeddings of multiple disjoint balls
$$
	w_{k}(Q; M):= \sup 
	\left\{R : \coprod_{i=1}^{k}\bB^{2n}_{R} \mbox{ embeds symplectically in $(M, \w)$ relative to $Q$}\right\}
$$
and this was undertaken in \cite{BC09, Ri10}. 
It is conceivable that our method for bounding $w_{1}(Q; M)$ could be adapted to get bounds for $w_{k}(Q; M)$ by using the triangle product on $CF^{*}(L; H)$ and having the auxiliary Lagrangian $L$ intersect $Q$ at the center of each ball.

Given two Lagrangians $Q$ and $L$ intersecting transversally, another ball packing problem considered
by Leclereq \cite{Le08} is symplectic embeddings $\i:(\bB^{2n}_{R}, \w_{0}) \to (M, \w)$ so that
$$
	\i^{-1}(Q) = \bB^{2n}_{R} \cap \R^{n} \quad
	\mbox{and}\quad  \i^{-1}(L) = \bB^{2n}_{R} \cap i\R^{n}.
$$
Let $w(Q,L; M)$ be the supremum over $R$ of such symplectic embeddings.  
Our method of studying $HF^{*}(L; H_{Q})$ also gives a method of finding bounds for $w(Q,L;M)$.

\subsubsection{The size of a Weinstein neighborhood}

A fundamental fact about a Lagrangian $Q \subset (M, \w)$ in a symplectic manifold is that it has a Weinstein 
\cite{We71} neighborhood, i.e.\ a neighborhood in $(M, \w)$ symplectomorphic to a neighborhood of the zero
 section in $(T^{*}Q, d\lambda_{Q})$, where $\l_{Q} = p\,dq$ in local coordinates.  
 Therefore one can wonder how large of a Weinstein neighborhood a
 given Lagrangian admits.  See \cite{El91, PPS03, Si89, Si91, Vi90b, Ze12} for work on this and similar questions.

One way to measure the size of a Weinstein neighborhood $\cN \subset M$ of a Lagrangian $Q \subset (M, \w)$
is as the width of $Q$ in $\cN$.
As the following proposition shows,
the width of a Lagrangian, which is a purely symplectic measurement of the Lagrangian in
the symplectic manifold, quantifies the maximal such size of a Weinstein neighborhood. 

\begin{prop}\label{p:sizeOfWeinstein}
	For a closed Lagrangian $Q \subset (M, \w)$ in a symplectic manifold
	$$
		w(Q; M) = \sup_{\cN} w(Q; \cN)
	$$
	where $\cN$ ranges over all Weinstein neighborhoods of $Q \subset (M, \w)$.
\end{prop}

This notion of the size of a Weinstein neighborhood also leads to a invariant for Riemannian manifolds in the following way.  Given a Riemannian metric $g$ on $Q$ one can define the Barraud--Cornea size
of $(Q, g)$ to be 
$$
	\mbox{S}_{BC}(Q, g) := w(Q; D^{*}_{g}Q)
$$
the width of $Q$ in the unit codisk bundle $D^{*}_{g}Q = \{ v \in T^{*}Q : \abs{v}_{g} \leq 1\}$.
This is a size-invariant in the sense of Guth \cite{Gu10} and it would be interesting to determine 
what it says about the Riemannian manifold $(Q, g)$. 

\begin{proof}[Proof of Proposition~\ref{p:sizeOfWeinstein}]
	The inequality $w(Q; M) \geq \sup_{\cN} w_{BC}(\cN)$ is by definition.
	For the opposite inequality we will show if $\i: \bB^{2n}_{R} \to (M, \w)$ is a symplectically embedding
	relative to $Q$, then for all $\e > 0$ there
	is a Weinstein neighborhood $\cN$ containing the image of $\i(\bB^{2n}_{R-\e})$.
	The proof is just a refinement of the Moser--Weinstein argument.
	
	Pick a compatible almost complex structure $J$ on $(M, \w)$ so that $\i^{*}J = J_{0}$ is the
	standard complex structure on the ball, which means the induced metric $g_{J}$ on $M$
	is such that $\i^{*}g_{J} = g_{0}$ is the standard Euclidian metric.  Define the map
	$$
		\Psi: T^{*}Q \to M \quad\mbox{by}\quad \Psi(v^{*}_{q}) = \exp_{q}(-J_{q}\Phi_{q}(v^{*}_{q}))
	$$
	where $\exp_{q}: T_{q}M \to M$ is the exponential map for $g_{J}$ and
	$\Phi_{q}: T^{*}_{q}Q \to T_{q}Q$ is the isomorphism induced by the metric $g_{J}$.
	For any choice of $J$ one has $\Psi_{*}d\l_{Q} = \w$ on vectors $T_{Q}M$
	and for our choice of $J$, 
	in Darboux coordinates on $T^{*}Q$ and $\i(\bB^{2n}_{R})$ induced by $\i$, one has
	$\Psi(x,y) = (x, -y)$ and hence $\Psi_{*}d\l_{Q} = \w$ on $\i(\bB^{2n}_{R})$.
	
	The homotopy $\Psi_{t}(v^{*}_{q}) = \Psi(t\, v^{*}_{q})$ 
	constructs a primitive $d\sigma = \w - \Psi_{*}d\l_{Q}$
	with $\sigma$ defined in an open neighborhood of $Q$ that contains
	$\i(\bB^{2n}_{R-\e})$.
	After restricting the domain to a neighborhood $\cN_{0}$ of the zero section, 
	Moser's method isotopes $\Psi$ to a symplectic embedding 
	$\widetilde{\Psi}: \cN_{0} \to (M, \w)$ 
	and since $\sigma$ vanishes on $\i(\bB^{2n}_{R-\e})$ one can ensure
	$\i(\bB^{2n}_{R-\e}) \subset \widetilde{\Psi}(\cN_{0})$.
\end{proof}

\subsubsection{A remark on notation}
In writing this paper it was necessary to use a non-trivial amount of loaded notation, so for the convenience of the reader at the end of the paper we have included a short index for the particularly subtle bits of notation used frequently within proofs. 


\subsection*{Acknowledgements}
We would like to thank Samuel Lisi and Antonio Rieser for introducing us to this problem and Octav Cornea for
his enthusiastic support and encouragement regarding this project.
We are very grateful to Helmut Hofer and Leonid Polterovich for 
their support with this work, and we would thank them as well as Mohammed Abouzaid, Paul Biran, Octav Cornea,
Lu\'{i}s Diogo, and Sobhan Seyfaddini for helpful conversations along with Georgios Dimitroglou Rizell
for sharing an preliminary draft of his paper \cite{Di13} with us.  
We would also like to thank the referees for their helpful comments and suggestions.


\section{The Lagrangian Floer--Hofer--Wysocki capacity}\label{s:LFHc}

\subsection{Lagrangian Floer cohomology}\label{ss:LFC}

Let us begin by briefly reviewing 
Lagrangian Floer cohomology \cite{Fl88, Fl88a, Oh93, Oh95} for admissible Lagrangians.
While these references restrict to compact Lagrangians $L$, due to the maximum principle in 
Lemma~\ref{l:MaxPrin} the results carry over to admissible Lagrangians.  While everything we say 
in this section is standard, in part we review it in order to establish our notations and conventions for the convenience of the reader.

\subsubsection{Preliminary definitions and Floer data}\label{s:Preliminaries}

Recall a \textbf{Liouville manifold} $(M^{2n}, \w)$ is an exact symplectic manifold $\w = d\th$
such that the vector field $Z_{\th}$, determined by $\i_{Z_{\th}}\w = \th$, has a complete flow 
$\vp_{Z_{\th}}^{t}$, and there is a compact codimension zero submanifold
$\overline{M} \subset M$ such that $Z_{\th}$ is positively transverse to $\d \overline{M}$ and 
$$
M = \overline{M} \cup \bigcup_{t\geq 0} \vp_{Z_{\th}}^{t}(\d\overline{M}).
$$
These conditions imply $\a := \th|_{\d\overline{M}}$ is a contact form on $\d\overline{M}$  
and there is an identification 
\begin{equation}\label{e:end}
M \backslash \Int\overline{M} = [1, \infty) \times \d\overline{M}
\end{equation}
given by the Liouville flow $\vp_{Z_{\th}}^{\log(r)}$
where for $r \in [1, \infty)$ one has
$\th = r\a$ and $Z_{\th} = r\d_{r}$.

\begin{defn}\label{d:admissibleLag}
We will be doing Floer theory with \textbf{admissible Lagrangians} $L \subset (M, d\th)$.  We define this to mean
$L$ is connected, orientable, and exact, i.e.\ $\th|_{L} = dk_{L}$ for some smooth $k_{L}: L \to \R$.  
If $L$ is not a closed manifold, then we will assume that $L$ is open, properly embedded in $M$, and 
$\supp(k_{L}) \subset L \cap \overline{M}$.
\end{defn}
Note that if we extend $k_{L}$ to a compactly supported function $k: M \to \R$, then
$\th' = \th - dk$ is still a Liouville $1$-form for the same symplectic
form.  Therefore for a fixed admissible Lagrangian $L$ we may assume $\th|_{L} = 0$ and $k_L = 0$.

\begin{defn}\label{d:cJ}
A compatible almost complex structure 
$J$ on a Liouville manifold $(M, d\th)$ is said to have \textbf{contact type} if 
\begin{equation}\label{e:Jct}
	\th \circ J = dr 
\end{equation}
on the cylindrical end \eqref{e:end}.  
The set of \textbf{admissible almost complex structures} $\cJ_{\th}(M)$ are smooth
	families of compatible almost complex structures $J = \{J_{t}\}_{t \in S^{1}}$ on $(M, d\th)$ that 
	at infinity are contact type and time independent. 
\end{defn}
The contact type condition implies $J$-holomorphic curves
$u: (S, j) \to (M, J)$ have a maximum principle, even when there are Lagrangian boundary conditions.
\begin{lem}[{\cite[Lemma 7.2]{AbS10}}]\label{l:MaxPrin}
	Let $(Y^{2n-1}, \xi)$ be a closed contact manifold with contact form $\a$
	and let $L$ be a properly embedded Lagrangian in $(W, d\th) = ([1, \infty) \times Y, d(r\a))$
	such that $\d L = L \cap \d W$ and $\th|_{L} = 0$.
	If $J$ is a compatible almost complex structure on $(W, d\th)$ with contact type
	and $(S, j)$ is a compact Riemann surface, then all $J$-holomorphic curves 
	$$u: (S, j) \to (W, J)$$ with $u(\d S) \subset \d W \cup L$ are constant.
\end{lem}
\begin{proof}
	It suffices to show that the $L^{2}$-energy $E(u) = 0$, where
	$E(u):= \tfrac{1}{2} \int_{S} \norm{du}^{2}$
	for the metric $d\th(\cdot, J\cdot)$.
	Using the fact that $u$ is $J$-holomorphic and $\th|_{L} = 0$, we get by Stokes' theorem,
	$$
		0 \leq E(u) = \int_{S} u^{*}(d\th)
		= \int_{\d_{n}S} u^{*}\th
	$$
	where $\d_{n}S \subset \d S$ is the part of the boundary mapped to $\d W$.
	If $\zeta \in T\d_{n}S$ is positively oriented, then $-j\zeta$ points outwards from $S$
	and hence it follows that $dr(du(-j\zeta)) \leq 0$.  
	Using that $u$ is $J$-holomorphic, i.e.\ $J\circ du = du \circ j$,
	and that $J$ has contact type \eqref{e:Jct} gives
	$\th(du(\zeta)) = dr(du(-j\zeta)) \leq 0$ and therefore $E(u) \leq 0$. Therefore $E(u) =0$ and hence
	$u$ is constant.
\end{proof}

\begin{defn}\label{d:cH}
	The set of \textbf{admissible Hamiltonians} $\cH \subset C^{\infty}(S^{1} \times M)$
	are those $H$ where there is a constant $M_{H}$ such that
	$H \leq M_{H}$, $\mbox{supp}(dH)$ is compact, and $H = M_{H}$ at infinity.  
\end{defn}

\subsubsection{The index and action of Hamiltonian chords}\label{sss:chords}

For an admissible Lagrangian $L \subset M$ and an admissible Hamiltonian $H$, let
$\cC^{*}_{H}(L)$ denote the \textbf{Hamiltonian chords for $L$}, i.e.\ the smooth paths
$x: [0,1] \to M$ where
\begin{equation}\label{e:LagChord}
	x(0), x(1) \in L \quad\mbox{and}\quad
	\tfrac{\d}{\d t}x(t) = X_{H_{t}}(x(t)).
\end{equation}
We will denote by $\cC_{H}(L) \subset \cC^{*}_{H}(L)$ the set of contractible chords, 
i.e.\ $[x] = 0$ in $\pi_{1}(M, L)$.  A chord $x \in \cC_{H}^{*}(L)$ is \textbf{non-degenerate} if the vector spaces 
$T_{x(1)}L$ and $d\vp_{H}^{1}T_{x(0)}L$ are transverse.

A \textbf{capping disk} $v$ of a chord $x \in \cC_{H}(L)$ is a smooth map
\begin{equation}\label{e:LagCap}
	v: \bD^{2} \to M \quad\mbox{such that $v(e^{\pi i t}) = x(t)$ and $v(e^{-\pi i t}) \in L$ for $t\in [0,1]$}.
\end{equation}
In Section~\ref{ss:maslovcap} we will recall how to associate a $\Z$-valued
Maslov index $\abs{(x,v)}_{\Mas}$ to a non-degenerate chord with a capping disk, and this
induces a well-defined $\Z/2$-grading for non-degenerate chords $x \in \cC_{H}(L)$
$$
	\abs{x}_{\Mas} := \abs{(x,v)}_{\Mas} \in \Z/2 \quad\mbox{for any capping disk $v$}.
$$
Finally there is an action functional
$\cA_{H,L}: \cC_{H}(L) \to \R$ defined by
\begin{equation}\label{e:LagAct}
	\cA_{H,L}(x) = \int_{0}^{1} H(t, x(t)) dt - \int_{0}^{1} x^{*}\th + k_{L}(x(1)) - k_{L}(x(0))
\end{equation}
where recall it is possible to pick $\th$ such that $\th|_L = 0$ and $k_L = 0$.
\begin{defn}\label{d:non-degenerate}
	An admissible Hamiltonian $H \in \cH$
	 is \textbf{non-degenerate with respect to $L$}
	if all chords $x \in \cC_{H}(L)$ with action $\cA_{H,L}(x) < M_{H}$
	are non-degenerate.
\end{defn}

\subsubsection{The complex}

Let $J \in \cJ_{\th}(M)$ be an admissible almost complex structure
and consider the Floer equation
\begin{equation}\label{e:FEL}
	\d_{s}u + J_{t}(u)(\d_{t}u - X_{H_{t}}(u)) = 0
\end{equation}
for smooth maps $u= u(s,t): \R \times [0,1] \to M$
that satisfy the boundary conditions $u(\R \times t) \subset L$ for $t=0,1$.
For a solution to \eqref{e:FEL} define its energy by
\begin{equation}\label{e:EnergyL}
	E(u) := \int_{\R \times [0,1]} \norm{\d_{s}u}^{2}_{J} ds\,dt 
	\quad\mbox{where $\norm{\d_{s}u}^{2}_{J} = d\th(\d_{s}u, J_{t}(u)\d_{s}u)$.}
\end{equation}
For non-degenerate chords $x_{\pm} \in \cC_{H}(L)$ let 
\begin{equation}\label{e:MLFC}
	\cM(x_{-}, x_{+}; L, H, J) 
\end{equation}
denote the set of finite energy solutions to the Floer equation
\eqref{e:FEL} that have asymptotic convergence 
$\lim_{s \to \pm \infty} u(s,\cdot) = x_{\pm}(\cdot)$.
Elements of $\cM(x_{-}, x_{+}; L, H, J)$ can be though of as negative gradient flow lines for $\cA_{H,L}$
and in particular there is the standard a priori energy bound
\begin{equation}\label{e:actionbound}
	0 \leq E(u) = \cA_{H,L}(x_{-}) - \cA_{H,L}(x_{+})
\end{equation}
for $u \in \cM(x_{-}, x_{+}; L, H, J)$.  
Note that since $X_{H_{t}}$ is compactly supported and $J$ is time independent and contact type
at infinity, the Floer equation \eqref{e:FEL} is the Cauchy-Riemann equation when $u$ is outside some compact set and therefore by Lemma~\ref{l:MaxPrin} solutions to \eqref{e:FEL} have a maximum principle.  

For fixed admissible $L$ and non-degenerate $H$, if the linearized operator of \eqref{e:FEL} is surjective for all 
$u \in  \cM(x_{-}, x_{+}; L, H, J)$, then $J$ is called \textbf{regular} for $(L, H)$ and such $J$ are generic
in $\cJ_{\th}$.
In particular the moduli space $\cM(x_{-}, x_{+}; L, H, J)$ is a smooth manifold,
whose dimension near a solution $u$ is 
\begin{equation}\label{e:dimModL}
\dim_{u} \cM(x_{-}, x_{+}; L, H, J) = 
\abs{(x_{-},v)}_{\Mas} - \abs{(x_{+},v \# u)}_{\Mas}
\end{equation}
where $v$ is any capping disk of the chord $x_{-}$ and $v \# u$ is the induced capping
disk for $x_{+}$.
Let 
$\cM_1(x_{-}, x_{+}; L, H, J)$ denote the union of the $1$-dimensional connected components of 
$\cM(x_{-}, x_{+}; L, H, J)$.  Translation in the domain gives an $\R$-action to the moduli space 
$\cM(x_{-}, x_{+}; L, H, J)$ and $\cM_1(x_{-}, x_{+}; L, H, J)/\R$ is a compact $0$-dimensional manifold.

For $a \in \R\cup\{\pm\infty\}$, let $CF^{*}_{a}(L;H)$ be the vector space over $\Z/2$ generated by chords
$x \in \cC_{H}(L)$ with action $\cA_{H,L}(x) > a$ and define the quotient
$$
CF_{(a,b]}^{*}(L;H) =
CF^{*}_{a}(L;H)/CF^{*}_{b}(L;H)
$$ 
where we will refer to $(a,b]$ as the action window.  This vector space
is $\Z/2$-graded whenever all chords in the action window are non-degenerate.
Standard compactness and gluing results show that if $H$ is non-degenerate with respect to
$L$ and $J$ is regular for $(L,H)$, then for
$b < M_{H}$ the $\Z/2$-linear map
\begin{equation}\label{e:diffL}
	d_{J}: CF^{*}_{(a,b]}(L;H) \to CF^{*+1}_{(a,b]}(L;H)	
\end{equation}
defined by counting isolated positive gradient trajectories
$$
	d_{J}x = \sum_{y} 
	(\#_{\Z_{2}} \cM_1(y, x; L, H, J)/\R) \, y
$$
where the sum is over chords $y \in \cC_{H}(L)$ with action in $(a,b]$,
makes $(CF^{*}_{(a,b]}(L;H), d_{J})$ a chain complex. 
Lagrangian Floer cohomology 
$$HF^{*}_{(a,b]}(L; H) = H^{*}(CF^{*}_{(a,b]}(L;H),d_{J})$$ is defined to be the homology
of this chain complex and since it is independent of the regular $J \in \cJ_{\th}$ we 
suppress it from the notation.

It follows from \eqref{e:actionbound} that the differential $d$ increases
the action $\cA_{H,L}$.  In particular when $a_{0} < a_{1}$ the inclusion
map
$CF^{*}_{(a_{1}, b]}(L;H) \to CF^{*}_{(a_{0}, b]}(L; H)$
is a map of chain complexes and induces
$$
	HF^{*}_{(a_{1}, b]}(L;H) \to HF^{*}_{(a_{0}, b]}(L;H).
$$
When $b_{0} < b_{1}$ the quotient map $CF^{*}_{(a,b_{1}]}(L;H) \to CF^{*}_{(a,b_{0}]}(L;H)$
induces a map
$$HF^{*}_{(a, b_{1}]}(L;H) \to HF^{*}_{(a, b_{0}]}(L;H).$$
These maps are called \textbf{action window maps}.

\subsubsection{Isomorphism with coholomogy}

\begin{defn}
	Let $L$ be an admissible Lagrangian in $(M, d\th)$
	and let $f: M \to \R$ be an admissible Hamiltonian that is non-degenerate with respect to $L$.
	If the following conditions are satisfied
	\begin{enumerate}
		\item[(i)] every chord $x \in \cC_{f}(L)$ is a critical point of $f|_{L}$,
		\item[(ii)]
		the only critical points for $\{f|_{L} > 0\}$ occur at infinity where $f$ is constant,
		\item[(iii)]
		the regular sublevel set $\{f|_{L} \leq 0\}$ is a deformation retract of $L$,
		\item[(iv)]
		$f|_{L}$ is a $C^{2}$-small Morse function on $\{f|_{L} \leq 0\}$,
	\end{enumerate}
	then we say $f$ is 
	\textbf{adapted} to $L$.
\end{defn}

It follows from \cite[Theorem 2]{Fl89a} that if $f:M \to \R$ is adapted to $L$
and $f > -a$,
then via Morse cohomology one has a chain-level isomorphism 
\begin{equation}\label{e:CLI}
	H^{*}_{\mbox{\tiny{Morse}}}(L) \cong HF^{*}_{(-a,0]}(L;f)
\end{equation}
given by mapping critical points $x \in \mbox{Crit}(f|_{L})$ with $f|_{L}(x) < 0$ to the corresponding constant chord 
$x \in \cC_{f}(L)$.  In particular
\begin{equation}\label{e:1L}
\mbox{$\ide_{L} \in H^{*}(L)$ corresponds to 
$\left[ \sum_{i=1}^{k} x_{i} \right] \in HF^{*}_{(-a,0]}(L;f)$}
\end{equation}
where $x_{i}$ are the critical
points of $f|_{L}$ with $f|_{L} \leq 0$ and Morse index zero.

\subsubsection{Continuation maps}

Let $(H^{-}, J^{-})$ and $(H^{+}, J^{+})$ be two regular pairs of admissible Hamiltonians non-degenerate with respect to $L$ and admissible almost complex structures.  For $s \in \R$ let $s \mapsto (H^{s}, J^{s})$ be a path of admissible Hamiltonians and almost complex structures that is constant at the ends
and connects the two original pairs $(H^{\pm}, J^{\pm}) = (H^{\pm \infty}, J^{\pm \infty})$.  
Consider solutions to the partial differential equation
\begin{equation}\label{e:ContMap}
	\begin{cases}
	\d_{s}u + J_{t}^{s}(u)(\d_{t}u - X_{H_{t}^{s}}(u)) = 0\\
	u: \R \times [0,1] \to M\\
	u(\R \times \{0,1\}) \subset L
	\end{cases}
\end{equation}
and for non-degenerate chords $x^{\pm} \in \cC_{H^{\pm}}(L)$, let
\begin{equation}
	\cM(x^{-}, x^{+}; L, \{H^{s}, J^{s}\}_{s})
\end{equation}
denote the set of finite energy solutions $u$ to \eqref{e:ContMap} such that 
$\lim_{s \to \pm\infty} u(s, \cdot) = x^{\pm}(\cdot)$.
When the path $J^{s}$ is generic, the spaces 
$\cM(x^{-}, x^{+}; L, \{H^{s}, J^{s}\}_{s})$ are finite dimensional manifolds
whose local dimension is given by \eqref{e:dimModL}.  

Let $\cM_{0}(x^{-}, x^{+}; L, \{H^{s}, J^{s}\}_{s})$ denote the zero dimensional components and consider
$$
	\Phi_{\{H^{s},J^{s}\}}: (CF^{*}_{(a,b]}(L; H^{+}), d_{J^{+}}) \to (CF^{*}_{(a,b]}(L; H^{-}), d_{J^{-}})
$$
which for $x^{+} \in \cC_{H^{+}}(L)$ with action in the window $(a,b]$ is defined by
$$
	\Phi_{\{H^{s},J^{s}\}}(x^{+}) = \sum_{x^{-}} \#_{\Z_{2}} \cM_0(x^{-}, x^{+}; L, \{H^{s}, J^{s}\}_{s}) \, x^{-}
$$
where the sum is over $x^{-} \in \cC_{H^{-}}(L)$ with action in $(a, b]$.
As with \eqref{e:actionbound}, for solutions to \eqref{e:ContMap} one has the bound
$$
	0 \leq E(u) \leq \cA_{H^{-}, L}(x^{-}) - \cA_{H^{+}, L}(x^{+}) + 
	\int_{\R \times [0,1]} (\d_{s}H_{t}^{s})(u(s,t))\,dsdt
$$
and hence if 
\begin{equation}\label{e:actionboundC}
	\int_{-\infty}^{+\infty} \sup_{M \times [0,1]}\d_{s}H_{t}^{s}\,ds \leq 0
\end{equation}
then $\Phi_{\{H^{s},J^{s}\}}$ preserves the action filtration, is a chain map,
and induces a map
$$
	\Phi_{\{H^{s},J^{s}\}}: HF^{*}_{(a,b]}(L; H^{+}) \to HF^{*}_{(a,b]}(L; H^{-})
$$
called a continuation map.

These maps are particularly nice when 
$H^{+} \leq H^{-}$ and the homotopy is monotone $\d_{s}H_{t}^{s} \leq 0$,
in which case we will call $\Phi_{\{H^{s},J^{s}\}}$ a \textbf{monotone continuation} map.
On homology 
monotone continuation maps are independent of the choice of monotone homotopy $(H^{s}, J^{s})$
used to define them, so we will denote them by
\begin{equation}\label{e:MonoCon}
	\Phi_{H^{+}H^{-}}:  HF^{*}_{(a,b]}(L; H^{+}) \to HF^{*}_{(a,b]}(L; H^{-}).
\end{equation}
They also commute with action window maps, and are natural in the sense that $\Phi_{HH} = \ide$
and
\begin{equation}\label{e:MCMN}
	\Phi_{H^{(2)}H^{(3)}} \circ \Phi_{H^{(1)}H^{(2)}} = \Phi_{H^{(1)}H^{(3)}}
\end{equation}
for admissible Hamiltonians $H^{(1)} \leq H^{(2)} \leq H^{(3)}$.

\subsection{The Lagrangian Floer--Hofer--Wysocki capacity}\label{s:LAS}

For a compact subset $X \subset M$, consider the set $\cH^{X}$ of
admissible Hamiltonians from Definition~\ref{d:cH} that are negative 
in a neighborhood of $X$ and are positive at infinity
\begin{equation}\label{e:cHW}
	\cH^{X} = \{ H \in \cH : 
	\mbox{
	$M_{H} > 0$ and
	$H|_{S^{1} \times X} < 0$.}	 \}.
\end{equation}
Since $(\cH^{X}, \leq)$ is a directed system, for $a > 0$ we define 
\begin{equation}
	HF^{*}(L; X, a) := \varinjlim_{H \in \cH^{X}} HF^{*}_{(-a, 0]}(L; H)
\end{equation}
where monotone continuation maps \eqref{e:MonoCon}
are used for the direct limit.
If $X_{-} \subset X_{+}$
are compact subsets, then there is a natural restriction map
\begin{equation}\label{e:LagRes}
	HF^{*}(L; X_{+},a) \to HF^{*}(L; X_{-}, a)
\end{equation}
since $\cH^{X_{+}} \subset \cH^{X_{-}}$.  If $a_{-} < a_{+}$, then the action window maps
induce a map
$$
	HF^{*}(L; X, a_{-}) \to HF^{*}(L; X, a_{+}).
$$   
For any compact subset $X \subset M$ there is a natural map
\begin{equation}\label{e:iLFH}
	i_{L;X}^{a}:  H^{*}(L) \to HF^{*}(L; X, a)
\end{equation}
given by the isomorphism \eqref{e:CLI} and 
the inclusion of $HF^{*}_{(-a,0]}(L;f)$ into the direct limit where $f \in \cH^{X}$ is
adapted to $L$ with $f > -a$.  

We now have the following definition where $\ide_{L} \in H^{0}(L)$ is the fundamental class.
\begin{defn}
	The \textbf{Lagrangian Floer--Hofer--Wysocki capacity} (relative to $L$) of $X$ is
	\begin{equation}\label{e:cL}
		c^{FHW}_{L}(X) = \inf\{ a > 0 : i_{L;X}^{a}(\ide_{L}) = 0\}
	\end{equation}
	where $c^{FHW}_{L}(X) = +\infty$ if 
	$i_{L;X}^{a}(\ide_{L}) \not= 0$ for all $a > 0$.
\end{defn}

Beyond Corollary~\ref{c:Main}, the other key property of $c_{L}^{FHW}$ we will use is the following lemma,
which follows directly from the definition of the direct limit.
\begin{lem}\label{l:critForLC}
	For any finite $a$, the capacity $c^{FHW}_{L}(X) \leq a$ if and only if there is
	an $f \in \cH^{X}$ adapted to $L$ and an $H \in \cH^{X}$ so that
	$-a < f \leq H$ and
	$$
		\ide_{L} \in \ker\left(\Phi_{fH} : HF^{*}_{(-a, 0]}(L;f) \to HF^{*}_{(-a,0]}(L; H)\right)
	$$
	where $\ide_{L} \in H^{*}(L) \cong HF^{*}_{(-a, 0]}(L;f)$ are identified as in \eqref{e:CLI}.	
\end{lem}




\section{Proving Theorem~\ref{t:main} with Lagrangian Floer cohomology} \label{section:mainproof}

For this section let $Q \subset (M^{2n}, d\th)$ be a closed orientable displaceable Lagrangian,
let $g$ be a Riemannian metric on $Q$, and let
\begin{equation}\label{e:REL}
\i: \bB_R^{2n} \to (M,d\th)
\end{equation}
be a symplectic embedding relative to $Q$.  For convenience we will fix a small parametrized Weinstein neighborhood of $Q$
\begin{equation}\label{e:FWN}
	\Psi: \{(q,p) \in T^{*}Q : \abs{p}_{g} < c\} \to M \quad\mbox{whose image we will denote $\cN$}
\end{equation}
and we will allow ourselves to decrease $c$ when we prove Lemma~\ref{l:L} below.
We will assume $\Psi(T^{*}_{\i(0)}Q) \subset \i(i\R^{n})$, that is $\Psi$ takes the cotangent fiber
$T^{*}_{\i(0)}Q$ into the image of the imaginary axis in the ball $\bB^{2n}_{R}$ under $\i$.  Under
our conventions if $\l_{Q} = p\,dq$ is the canonical $1$-form on $T^{*}Q$, then the Hamiltonian flow for
$\tfrac{1}{2}\abs{p}_{g}^{2}$ in $(T^{*}Q, d\l_{Q})$ is the cogeodesic flow.

\subsection{Geodesic paths in $Q$ via an auxiliary Lagrangian $L$}

We will be doing Lagrangian Floer cohomology for an auxiliary admissible Lagrangian $L$ of the following form,
whose existence we will establish in Section~\ref{s:auxLag}.

\begin{lem}\label{l:L}
	There is an admissible Lagrangian $L \subset (M, d\th)$ such that:
	\begin{itemize}
	\item[(i)] $L$ is diffeomorphic to $\R^{n}$ and displaceable from $Q$.
	\item[(ii)] $L$ intersects the ball only along the imaginary axis, i.e.\ $\i^{-1}(L) = i\R^{n} \cap \bB^{2n}_{R}$.
	\item[(iii)] $L$ intersects $Q$ along cotangent fibers in the Weinstein neighborhood $\cN$ of $Q$
	\begin{equation}\label{e:NcL}
		\cN \cap L = \bigcup_{j=0}^{k} \Psi(T^{*}_{q_{j}}Q)
	\end{equation}
	where $Q \cap L = \{q_{0}, q_{1}, \dots, q_{k}\}$ with $q_{0} = \i(0)$ and $k \geq 1$.
\end{itemize}
\end{lem}
From now on we will fix such an auxiliary admissible Lagrangian $L \subset (M, d\th)$
diffeomorphic to $\R^{n}$, and we will assume both that $\th|_{L} = 0$ and $k_L = 0$.

\subsubsection{A family of Hamiltonians $\cH_{\cN}^{Q}$ and their chords}\label{s:cHcN}

We will now introduce a special family of Hamiltonians, depicted in Figure~\ref{f:cHcN}, that are specially adapted
to the Weinstein neighborhood $\cN$.

\begin{defn}\label{d:cHcN}
	Define $\cH_{\cN\!,g}^{Q} \subset \cH^{Q}$ to be those admissible Hamiltonians $H: M \to \R$
	that are constant outside of $\cN \subset M$ and inside $\cN$ have the form $H = f_{H}(\abs{p}_{g})$
	for a smooth function $f_{H}: \R \to [-\e_{H}, \infty)$ where $0 < \e_{H} \ll 1$ and for positive constants 
	$i_{H} < \rho_{H} \leq c$
	satisfies the following conditions for $r \geq 0$:
	\begin{enumerate}
	\item[(1)] $f_H'(r) \geq 0$.	
	\item[(2)] $f_{H}(0) = -\e_{H}$, $f_{H}'(0) = 0$, and $f_{H}''(0) > 0$.
	\item[(3)] $f_{H} = M_{H}$ is a positive constant in an open neighborhood of where $r \geq \rho_{H}$.
	\item[(4)] If $r \leq i_H$, then $f_H''(r) \geq 0$.  While
	$f_H''(r) \leq 0$ for $r \geq i_H$.
	\item[(5)] $f_H''(r) \neq 0$ if $f_H'(r) \neq 0$
	and $r \neq  i_H$. 
	\end{enumerate}
	If the particular metric is not important to us we will suppress the $g$ and just write $\cH_{\cN}^{Q}$.	
\end{defn}
We will call $f_{H}'(i_{H})$ the \textbf{slope} of $H$ and we will assume that this is not equal
to the length of a geodesic path in $Q$ connecting points in $Q \cap L$.  
Note that there will be degenerate
constant chords on $L$ where $H \equiv M_{H}$, but since
their action $\cA_{H,L} = M_{H}$ is positive they will
not appear in the complex $CF^{*}_{(-\infty, 0]}(L;H)$.  Henceforth when $H \in \cH_{\cN}^{Q}$ we will only speak of chords $x \in \cC_{L}(H)$ with action $\cA_{H,L}(x) < M_{H}$.
\begin{figure}
\begin{center}
\begin{tikzpicture}[domain=0:4]
   \draw[->] (-0.2,0) -- (8.2,0) node[right] {$r$};
   \draw[->] (0,-1.2) -- (0,5) node[above] {$f_H(r)$};
    \draw (0,-.1) to[out=0,in=225] (3.5,2);
    \draw (3.5,2) to[out=45,in=180] (7,4);
    \draw (7,4) to (8,4);
    \draw (3.5,0.1) -- (3.5,-0.1) node[below] {$i_H$};
    \draw (7,0.1) -- (7,-0.1) node[below] {$\rho_H$};
    \draw (0.1,4) -- (-0.1,4) node[left] {$M_{H}$};
\end{tikzpicture}
\end{center}
\caption{We use Hamiltonians $H(q,p) = f_{H}(\abs{p}_{g})$ in the Weinstein neighborhood $\cN$.}
\label{f:cHcN}
\end{figure}
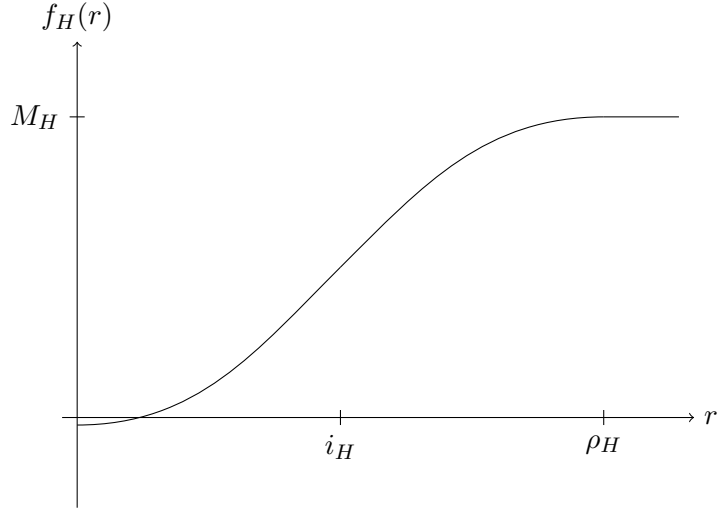

Given $H \in \cH_{\cN\!,g}^{Q}$, its chords $x(t) = (q(t), p(t)) = \vp_{H}^{t}(q(0), p(0))$
are such that $q: [0, 1] \to Q$ are constant speed geodesics with respect to $g$, where 
$\abs{\dot{q}}_{g} = f_{H}'(\abs{p(t)}_{g})$,
with endpoints $q(0), q(1) \in Q \cap L$.
We define the \textbf{cotangent bundle action} of such chords to be
\begin{equation}\label{e:CTBA}
	\cA_{H,L}^{T^{*}Q}(x) = \int_{0}^{1} H(x(t)) dt - \int_{0}^{1} x^{*}\l_{Q}
\end{equation}
and for $x(t) = (q(t), p(t))$ we have the identity
\begin{equation}\label{e:CTBAI}
	\cA_{H,L}^{T^{*}Q}(x) = f_{H}(\abs{p}_{g}) - f_{H}'(\abs{p}_{g})\abs{p}_{g}.
\end{equation}
Equation (\eqref{e:CTBAI}) tells us that the cotangent bundle action of $x$ can be identified with the $y$-intercept
of the tangent line to the graph of $f_H$ at $\abs{p}_{g}$.
Hence it is easy to see that one has the bound
\begin{equation}\label{e:CTBAB}
	\cA_{H,L}^{T^{*}Q}(x) \geq B_{f_{H}}:=  f_{H}(i_{H}) - f_{H}'(i_{H}) i_{H}
\end{equation}
for all chords. Furthermore for a fixed slope $\l$, and constants $M_{H}$
and $\rho_{H}$, the bound $B_{f_{H}}$ can be made arbitrarily close to zero by requiring $f_{H}$ to be
$C^{0}$-close to a piecewise linear function with slope $\l$ near $r=0$ and is the constant $M_{H}$ 
when $r \geq \rho_{H}$.

Any non-constant geodesic path $q : [0,1] \to Q$ with endpoints in $Q \cap L$, with constant speed less
than the slope of $H$, and that is zero in $\pi_{1}(M, L)$, appears exactly twice as a chord:  
Let $\l_{n} < \l_{f}$ be the unique positive numbers
such that $f_{H}'(\l_{n}\abs{\dot{q}}_{g}) = f_{H}'(\l_{f}\abs{\dot{q}}_{g}) = \abs{\dot{q}}_{g}$, then
for $p_{0} = g(\dot{q}(0), \cdot)$
$$
	x_{n}(t) = \vp_{H}^{t}(q(0), \l_{n}p_{0}) \quad\mbox{and}\quad x_{f}(t)= \vp_{H}^{t}(q(0), \l_{f}p_{0})
$$
are both chords for $H$ that represent the geodesic path $q$.  We will call $x_{n}$ the \textbf{near chord} and $x_{f}$ the \textbf{far chord}.
 
The complex
$CF^{*}_{(-\infty, 0]}(L; H)$ is generated by chords of the following type:
\begin{itemize}
\item \textbf{Constant chords}: The points $q_{i} \in Q \cap L$. They have index $\abs{q_{i}}_{\Mas} = 0$ and action $\cA_{H,L}(q_{i}) = -\e_{H}$.
\item \textbf{Near chords}: In the region where $f_{H}'' > 0$ and $f_{H}' > 0$.
\item \textbf{Far chords}: In the region where $f_{H}'' < 0$ and $f_{H}' > 0$.
\end{itemize}
We will also introduce the following dichotomy for chords, illustrated in Figure~\ref{f:LoopPath}.
\begin{itemize}
\item \textbf{Path chords}: Chords whose corresponding geodesic $q$ is such that $q(0) \not= q(1)$.
\item \textbf{Loop chords}: Chords whose corresponding geodesic $q$ is such that $q(0) = q(1)$.
\end{itemize}
This dichotomy will play an important role when we try to control the behavior of differentials in the relatively
embedded ball $\bB^{2n}_{R}$.

\begin{figure}[h]
  \centering
  \def\svgwidth{300pt}
  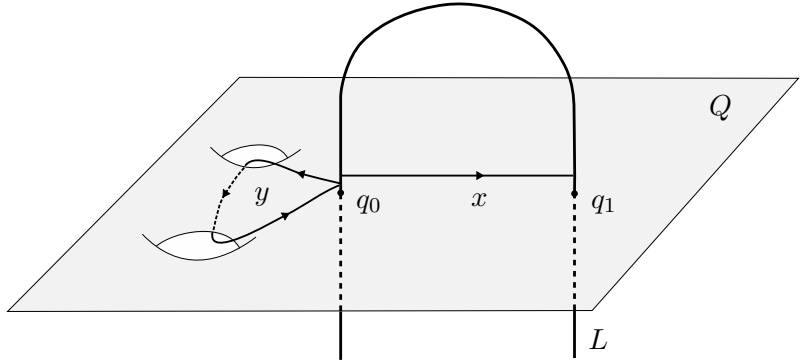
  \caption{A path chord $x$ representing a geodesic from $q_{0}$ to $q_{1}$ and a loop chord
  $y$ representing a geodesic starting and ending at $q_{0}$.}
  \label{f:LoopPath}
\end{figure}

\subsubsection{The Lagrangian capacity $c^{FHW}_{L}(Q)$ and Hamiltonians in $\cH_{\cN}^{Q}$}

By shifting $H \in \cH_{\cN}^{Q}$ up slightly we may always assume that
we can find an $\e > \e_{H}$ small enough so that the constant chords $q_{i} \in Q \cap L$
span $CF^{*}_{(-\e, 0]}(L;H)$.  For action reasons, this means that 
each intersection point $q_{j} \in CF^{*}_{(-\e, 0]}(L;H)$ is a cycle and we have the following
lemma.
\begin{lem}\label{l:inclusionSpecialH}
	Let $f \in \cH^{Q}$ be a $C^{2}$-small Hamiltonian adapted to $L$ such that
	$-\e < f \leq H$.
	For any $a \geq \e$, the monotone continuation map
	$\Phi_{fH}: H^{*}(L) \to HF^{*}_{(-a,0]}(L;H)$
	is such that $\Phi_{fH}(\ide_{L}) = \left[\sum_{j=0}^{k} q_{j}\right]$,
	where we used the identification \eqref{e:1L}.
\end{lem}
\begin{proof}
	By the naturality of monotone continuation maps it suffices to prove this for a particular $f$,
	so pick $f$ to be such that $f = H$ in a neighborhood of the points $\{q_{0}, \dots, q_{k}\} = Q \cap L$
	and all local minima of $f|_{L}$ have value at least $-\e_{H}$.  If one picks a monotone homotopy
	between $f$ and $H$ that is constant near the points $q_{j}$,
	then $\Phi_{fH}(q_{j}) = q_{j}$ on the chain level
	$\Phi_{fH}: CF^{*}_{(-a, 0]}(L; f) \to CF^{*}_{(-a, 0]}(L; H)$
	since there is only the constant solution for energy and action reasons.
	The result now follows from \eqref{e:1L}.
\end{proof}

With this lemma, the energy-capacity inequality in Corollary~\ref{c:Main} 
and Lemma~\ref{l:critForLC} give the following proposition.
\begin{prop}\label{p:thediff}
	For any finite $a > e(Q; M)$, there is an $H \in \cH^{Q}_{\cN}$ such that
	for any admissible regular 
	$J \in \cJ_{\th}(M)$
	there is a chord $x \in \cC_{H}(L)$ such that
	$\ip{d_{J}x, q_{0}} \not= 0$ in $(CF^{*}(L;H), d_{J})$
	with action $\cA_{H,L}(x) > -a$.  In particular there is a differential
	$$
		u \in \cM(q_{0}, x; L, H, J) \quad\mbox{for the moduli space in \eqref{e:MLFC}}
	$$
	with the energy bound $E(u) \leq a$.  This continues to be true for any
	$H^{+} \in \cH^{Q}_{\cN}$ with $H^{+} \geq H$.	
\end{prop}
\begin{proof}
By Corollary~\ref{c:Main} and Lemma~\ref{l:critForLC} there is an $H \in \cH^{Q}_{\cN}$
so that $$\Phi_{fH}(\ide_{L}) = 0 \in HF^{0}_{(-a, 0]}(L;H).$$
By Lemma~\ref{l:inclusionSpecialH}, this means the cycle
$\sum_{j=0}^{k} q_{j} \in CF^{*}_{(-a,0]}(L;H)$ is a boundary
and hence there is a chord $x \in CF^{-1}_{(-a, 0]}(L;H)$
with $\ip{d_{J}x, q_{0}} \not= 0$.
\end{proof}

Note that Proposition~\ref{p:thediff} remains true if we insist that the almost complex structure $J$
has a particular form on the ball $\i(\bB^{2n}_{R})$.  This is because no differential is contained entirely in
the ball and hence regularity can still be achieved among such almost complex structures. 
In particular we can assume $J \in \cJ_{\iota}(V)$,
which is defined as follows.
\begin{defn}\label{d:JinBall}
For a subset $V \subset \bB^{2n}_{R}$ define $\cJ_{\iota}(V) \subset \cJ_{\th}(M)$ to be the subset of admissible 
almost complex structures $J$ in Definition~\ref{d:cJ} such that $J|_{\i(V)} = \i_{*}J_{0}$ where $J_{0}$ is
the standard complex structure on $\C^{n}$ and $\i$ is the relative ball embedding \eqref{e:REL}.
\end{defn}

Our goal is to use a differential as in Proposition~\ref{p:thediff} to build a certain holomorphic
curve in the relatively embedded ball in order to prove Theorem~\ref{t:main}.  
It is at this point where we will bring in the assumption that $Q$ has a metric with non-positive sectional curvature
in order to prove the following theorem.  For this theorem let $U$ be any neighborhood of 
$\R^{n} \cap \bB^{2n}_{R}$ and $\cN$ be a displaceable
Weinstein neighborhood of $Q$ of the form \eqref{e:FWN} where $\i^{-1}(\cN) \subset U$.

\begin{thm}\label{t:ExistsPCD}
	Let $g$ be a metric of non-positive curvature on a Lagrangian $Q$ as in Theorem~\ref{t:main}.
	For any finite $a > e(Q;M)$, there is a Hamiltonian $H \in \cH_{\cN\!,g}^{Q}$ and a 
	$J \in \cJ_{\i}(\bB^{2n}_{R} \backslash U)$ such that
	there is an element $u \in \cM(q_{0}, x; L, H, J)$ 
	of the moduli space \eqref{e:MLFC} with energy $E(u) \leq a$ that connects
	$q_{0} = \i(0)$ to a path chord $x \in \cC_{H}(L)$.
\end{thm}

The main content here is that we can use 
the non-positive curvature assumption to strengthen the conclusion of 
Proposition~\ref{p:thediff} so that the given differential involves a path chord $x$.  
The fact that we can take $x$ to be a path chord plays a key role in our proof of 
Theorem~\ref{t:main}.
We will prove Theorem~\ref{t:ExistsPCD} in Section~\ref{s:pathchordexistence}.

\subsection{Using Floer differentials to prove Theorem~\ref{t:main}}\label{ss:limit}

We will now use Theorem~\ref{t:ExistsPCD} to prove Theorem~\ref{t:main} as outlined in Section~\ref{ss:CompIdea}.
In what follows the interior of a set $Y \subset \R^{2n}$ will be denoted $\mathring{Y}$.

\begin{proof}[Proof of Theorem~\ref{t:main}]
Without loss of generality we may assume we have a relatively embedded ball
$\bB^{2n}_{R+\epsilon}$ for some very small $\e>0$ so that the preimage of our Lagrangians under $\iota$ are still linear.
Consider the sequence of neighborhoods 
$$
U_{k} = \left\{z \in \bB^{2n}_{R+\e} : \sum_{i=1}^{n} \abs{\text{Im}(z_{i})} < 1/k \right\} \quad
\mbox{of $\R^{n} \cap \bB^{2n}_{R+\e}$}
$$
and corresponding Weinstein neighborhoods $\cN_{k}$, Hamiltonians $H_{k}$, almost
complex structures $J_{k}$, path chords $x_{k}$, and differentials $u_{k}$ given by Theorem~\ref{t:ExistsPCD}. 
For any $e > e(Q;M)$, we may assume $E(u_{k}) \leq e$ for all $k$.

Observe that if $x \in \cC_{H}(L)$ is a path chord, then $\{x(0), x(1)\} \not\subset \iota(\bB^{2n}_{R+\e})$ since otherwise 
for $x(t) = (q(t), p(t))$ one has $q(0) = q_{0} = q(1)$.  Therefore by the boundary conditions
on elements in the moduli space $\cM(q_{0}, x_{k}; L, H_{k}, J_{k})$ it follows that the image of 
$u_{k}(\R \times \{0,1\})$ contains
a path $\gamma_{k}$ in $L \cap \i(\bB^{2n}_{R+\e})$ from $q_{0}$ to a point on $\i(\d\bB_{R+\e}^{2n})$,
see Figure~\ref{f:boundaryC}.
\begin{figure}[h]
  \centering
  \def\svgwidth{300pt}
  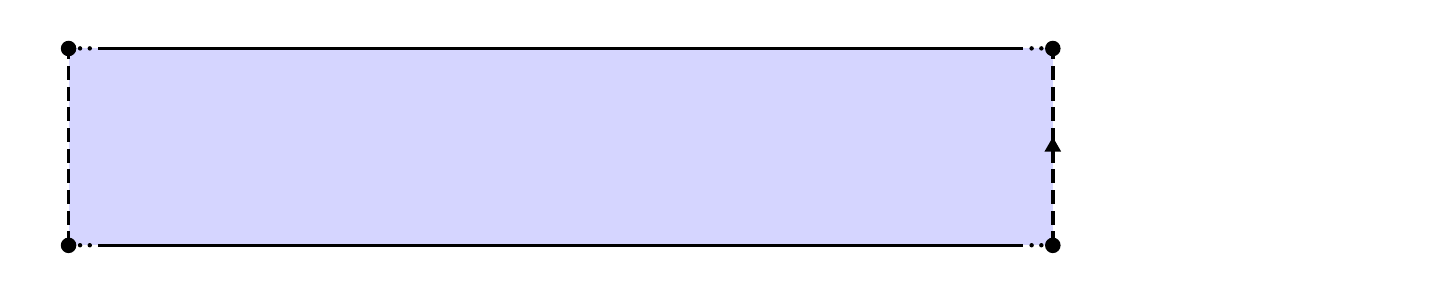
  \caption{
 If $x(1) \notin \iota(\bB^{2n}_{R+\e})$, then $u(\cdot,1): [-\infty, b] \to L \cap \i(\bB^{2n}_{R+\e})$ 
  is a path from $q_{0}$ to a point on $\iota(\d\bB^{2n}_{R+\e})$ for some $b \in \R$.}
  \label{f:boundaryC}
\end{figure}

From now on we will focus on $\iota(\bB^{2n}_{R+\e})$, so by composing with $\iota^{-1}$ we will view $u_{k}$ as a map to 
$\bB^{2n}_{R+\e}$ and $\g_{k}$ as a path in $i\R^{n} \cap \bB^{2n}_{R+\e}$ from $0$ to the boundary.
Pick a sequence $V_k \subset \bB^{2n}_{R+\e} \setminus U_{k}$ of compact codimension $0$ submanifolds such that
$$
\bigcup_{k=1}^{\infty} V_k = \bB^{2n}_{R+\e} \setminus \R^n \quad\mbox{and}\quad
V_{k-1} \cap \mathring{\bB}^{2n}_{R+\e} \subset \mathring{V}_k.
$$
If we choose $V_k$ so that $i\R^{n} \cap \d\bB_{R+\e}^{2n} \subset V_{k}$, then for each $j,k \geq 0$
by the intermediate value theorem we have
$\gamma_{k} \cap \partial V_j$ is non-empty.
For each $j \leq k$, pick a point $l_{j,k} \in \gamma_{k} \cap \partial V_j$ with minimum distance to $0$
and we have $l_{j,k} \to 0$ uniformly as $j \to \infty$.
\begin{figure}[h]
  \centering
  \def\svgwidth{400pt}
  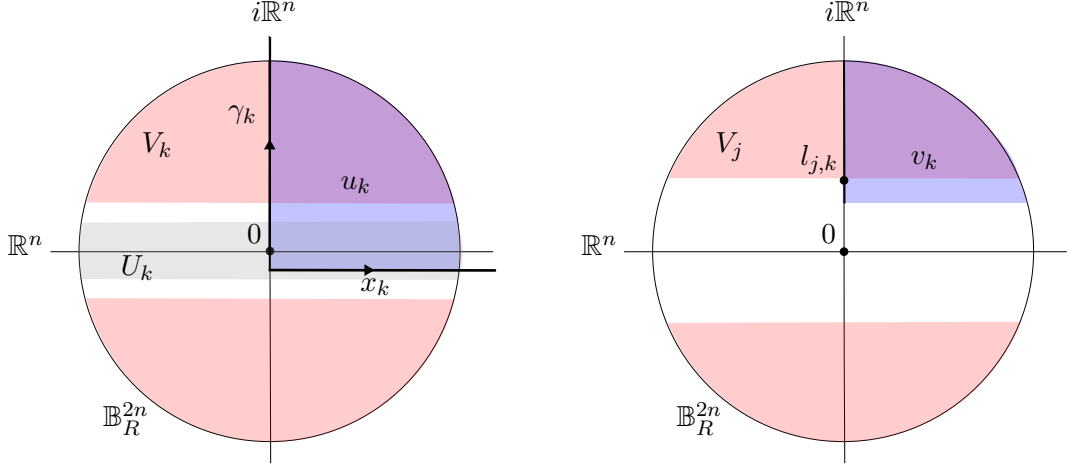
  \caption{The sets $V_{k}$ are in red and the set $U_{k}$ is in grey.
  Left: The setting of the proof of Theorem~\ref{t:main}.  The image of the differential $u_{k}$ is in blue
  and the $J_{0}$-holomorphic curve $v_{k}: \Si_{k} \to V_{k}$ is where $u_{k}$ maps to $V_{k}$.
  Right: The setting of Lemma~\ref{l:existenceofproperJholmap} where the image of the $J_{0}$-holomorphic
  curve $v_{k}$ is in blue.}
  \label{f:FirstLimit}
\end{figure}
Because
$\supp(dH_{k}) \subset \cN_{k}$ and $J_{k} \in \cJ_{\i}(\bB_{R+\e}^{2n}\backslash U_{k})$
it follows that the part of the differential $u_{k}$ that is 
in $\mathring{V}_k$ can be seen as a proper holomorphic curve
$$
	v_{k} = u_{k}|_{u_{k}^{-1}(\mathring{V}_k)}: \Sigma_{v_{k}} \to 
	(\mathring{V}_k, J_{0})
$$
with energy $E(v_{k}) \leq e$.
See Figure~\ref{f:FirstLimit} for a schematic drawing.

By Lemma \ref{l:existenceofproperJholmap} below, using our holomorphic maps $v_k$ 
we get a proper holomorphic map 
$$v: \Sigma \to (\mathring{\bB}^{2n}_{R}, J_{0})$$ 
passing through $0$ with energy $E(v) \leq 4e$, where $\Sigma$ is a Riemann surface without boundary. 
By the standard monotonicity estimate (e.g.\ \cite[Section 4.3]{Si94}), 
the holomorphic curve $v$ has energy at least $R \leq E(v)$ and hence 
$R \leq 4e$.  Since this holds for all $e > e(Q; M)$ and $R < w(Q;M)$, Theorem~\ref{t:main} follows.
\end{proof}

\begin{lem}\label{l:existenceofproperJholmap}
Let $V_{k} \subset \bB^{2n}_{R}$ be a sequence of compact codimension $0$ submanifolds with boundary
with the property that $V_k \cap \mathring{\bB}^{2n}_{R} \subset \mathring{V}_{k+1}$ and 
$\cup_k V_k = \bB^{2n}_{R} \setminus \R^n$.
Let $v_k : \Sigma_k \to (\mathring{V}_k, J_{0})$ be a sequence of proper holomorphic maps from genus $0$ Riemann surfaces with uniform energy bound $E(v_{k}) \leq e$.
Suppose also that $\d\Sigma_k$ gets mapped to $i\R^n \subset \bB^{2n}_{R}$
and that for all $j < k$ there exists $l_{j,k} \in \text{image}(v_{k}|_{\d\Sigma_{k}}) \cap \partial V_{j}$ 
with $l_{j,k} \to 0$ uniformly as $j \to \infty$.
It then follows there is a proper holomorphic map $v : \Sigma \to (\mathring{\bB}^{2n}_{R}, J_{0})$
with energy $E(v) \leq 4e$ passing through $0$ where $\Sigma$ is a Riemann surface without boundary.
\end{lem}
\begin{proof}[Proof of Lemma \ref{l:existenceofproperJholmap}]
Let $\sigma : \bB^{2n}_{R} \to \bB^{2n}_{R}$ be the map sending each complex coordinate $x+iy$
to $-x + iy$. We replace $V_k$ with an appropriate smoothing of $V_k \cap \sigma(V_k)$ and restrict our holomorphic curves to this smaller manifold so that $V_k$ is invariant under the action of $\sigma$.
By the Schwarz reflection principle we can reflect $v_k$ along $i \R^n$
via $\sigma$ and create a new proper holomorphic map $v_k : \Sigma_k \to \mathring{V}_k$
where $\Sigma_k$ is an open Riemann surface without boundary 
and $E(v_{k}) \leq 2e$.
We can assume that the boundaries of $V_j$ are generic enough so that
$v_k$ is transverse to $\partial V_j$ for $j < k$
and $\Sigma_{k,j} := v_k^{-1}(V_j)$ is a compact submanifold with boundary for $j < k$.

Fish's compactness result \cite[Theorem A]{Fi11} tells us that for fixed $j$ there is
\begin{enumerate}
\item[(1)] a compact Riemann surface $S_{j}$ with boundary and 
a compact nodal Riemann surface $S_{j}'$ with boundary with a surjective continuous
map $\phi_j : S_j \to S'_j$,
\item[(2)] smooth embeddings $\phi_{k,j} : S_j \to \Sigma_{k,j}$ such that $v_{k} \circ \phi_{k,j}(\d S_{j}) \subset V_{j}\backslash V_{j-1}$,
\item[(3)] a $J_{0}$-holomorphic map  $u_j : S'_j \to V_{j}$ with energy at most $2e$ 
where $u_{j}(\d S_{j}') \subset V_{j}\backslash V_{j-1}$,
\end{enumerate}
and a subsequence of $\{v_{k}|_{\Sigma_{k,j}}\}_{k}$ such that
$$
	\mbox{$v_k \circ \phi_{k,j}: S_{j} \to V_{j}$ converges $C^{0}$-uniformly to $u_j \circ \phi_j: S_{j} \to V_{j}$.}
$$
By a diagonal process of successive subsequences we can ensure that the image of $u_{j}$ is contained in the image of $u_{j+1}$, and let $A$ be the union of these images noting that its energy is bounded by $2e$.
We also have $l_{j,k}$ converges to some point $l_j$ in $\d V_j$ as $k \to \infty$ and that $l_j \to 0$
because $l_{j,k} \to 0$ uniformly as $j \to \infty$. Hence the closure of $A$ contains $0$ because $l_j$ is in the image of $u_{j}$.
The union of $A$ with its complex conjugate
$A \cup \overline{A}$
is a closed analytic subvariety of complex dimension $1$ in $\bB^{2n}_{R} \setminus \R^n \subset \C^n$
which is invariant under complex conjugation.
If $X$ is the closure of $A \cup \overline{A}$ inside $\bB^{2n}_{R}$, then by
the main theorem in \cite{Al71} $X$ is a closed analytic subvariety of $\bB^{2n}_{R}$.
Let $v : \widetilde{X} \to X \subset \bB^{2n}_{R}$ be the normalization of $X$, this is a proper holomorphic map from a Riemann surface $\Sigma := \widetilde{X}$ with boundary to 
$\bB^{2n}_{R}$ of energy at most $4e$ passing through $0$ such that $v(\d\Sigma) \subset \d\bB^{2n}_{R}$.
\end{proof}
See \cite[Chapter 8]{GrRe84} for background on normalization for analytic spaces and \cite[Chapter 6.5]{GrRe84} for the proof that $1$-dimensional normal complex spaces are Riemann surfaces.

\subsection{Building the auxiliary Lagrangian $L$} \label{s:auxLag}

We will now present the construction of the auxiliary Lagrangian $L$ from Lemma~\ref{l:L}.
As a first step we have a local construction of a Lagrangian $\R^{n}$ on the cylindrical end.

\begin{lem}\label{l:LatI}
	If $(Y^{2n-1}, \xi)$ is a closed contact manifold with contact form $\a$, then
	there is a properly embedded Lagrangian $L$ in
	$([1,\infty) \times Y, d(r\a))$ diffeomorphic to $\R^{n}$ with $(r\a)|_{L} = dh_{L}$
	for a smooth compactly supported $h_{L}: L \to \R$.
\end{lem}
\begin{proof}
For $\th_{0} = \tfrac{1}{2}\sum_{i=1}^{n} x_{i}dy_{i} - y_{i}dx_{i}$ in $\C^{n}$ consider the
standard contact structure $(S^{2n-1}, \xi_{0})$ with $\xi = \ker \a_{0}$ where $\a_{0} = \th_{0}|_{S^{2n-1}}$
and the exact symplectic embedding
$$
	\Phi: ((0, \infty)\times S^{2n-1}, d(r\a_{0})) \to (\C^{n}, d\th_{0}) \quad\mbox{by}\quad (r,x) \mapsto rx.
$$
By the contact Darboux theorem for a sufficiently small open set $U \subset Y^{2n-1}$ there is an open set 
$V \subset (S^{2n-1}, \xi_{0})$ 
containing $\R^{n} \cap S^{2n-1}$ and a contactomorphism 
$\psi: (V, \xi_{0}) \to (U, \xi)$ such that $\psi^{*}\a = f\, \a_{0}$ for some $f: V \to (0,\infty)$.
By shrinking $U$ and $V$ slightly we can assume that $f \geq m_{f} > 0$ where $m_{f}$ is a constant and define
the exact symplectic embedding
\[ 
\Psi : ([m_{f},\infty) \times V, d(r\a_{0})) \to ([1,\infty) \times U, d(r\a)) \quad\mbox{by}\quad
\Psi(r,x) = \left(\tfrac{r}{f(x)}, \psi(x)\right).
\]
Since $V \subset S^{2n-1}$ contains $\R^{n} \cap S^{2n-1}$,
we can use a compactly supported Hamiltonian diffeomorphism in $(\C^{n}, d\th_{0})$
to move the Lagrangian $\R^{n} \subset \C^{n}$ into the image $\Phi([m_{h},\infty) \times V) \subset \C^{n}$. 
The image of this new Lagrangian under $\Psi$ in $([m_{f},\infty) \times V, d(r\alpha_{0}))$ is our desired Lagrangian.
\end{proof}

We can now prove Lemma~\ref{l:L}.

\begin{proof}[Proof of Lemma~\ref{l:L}]
It follows from Lemma~\ref{l:LatI} that any Liouville manifold $(M, d\th)$ 
contains an admissible Lagrangian $L$ diffeomorphic to $\R^{n}$.
By an appropriate compactly supported Hamiltonian diffeomorphism of $(M, d\th)$ we can
assume for some $\e > 0$ that the Lagrangian $L$ is such that
$$
	\i(0) \in Q \cap L \quad\mbox{with}\quad \i^{-1}(L) \cap \bB^{2n}_{\e} = i\R^{n} \cap \bB^{2n}_{\e}.
$$
If we modify $\th$ to $\th'$ by adding a compactly supported exact $1$-form so that
$\i^{*}\th' = \th_{0}$ in $\bB^{2n}_{R}$, then 
the Liouville vector field $X_{\th'}$ will have the form $X_{\th'} = \tfrac{1}{2}\sum_{i=1}^{n} x_{i}\d_{x_{i}} + y_{i}\d_{y_{i}}$ in the
ball $\i(\bB^{2n}_{R})$.  Flowing $L$ along $X_{\th'}$ gives a new Lagrangian $L$ such that 
$\i^{-1}(L) = i\R^{n} \cap \bB^{2n}_{R}$.  By applying another compactly supported Hamiltonian symplectomorphism to $L$ we can get a new Lagrangian $L$ such that $\i^{-1}(L) = i\R^{n} \cap \bB^{2n}_{R}$
still holds and for some $c > 0$ sufficiently small in the Weinstein neighborhood \eqref{e:FWN}
$$
	\Psi^{-1}(L) \cap \{\abs{p}_{g} < c\} = \bigcup_{j=0}^{k} T^{*}_{q_{j}}Q \cap \{\abs{p}_{g} < c\}
$$ 
where $q_{0} = \i(0)$ and $\{q_{0}, \dots, q_{k}\} = Q \cap L$.  

To show that $Q$ and $L$ do not only intersect at $\i(0)$, we will show that they intersect an
even number of times.
By construction $L$ can be made disjoint from $Q$
by a Hamiltonian isotopy, so it follows that under the intersection product
$$
	\cap: H_{n}^{\rm{lf}}(M) \otimes H_{n}(M) \to H_{0}(M) \quad\mbox{that}\quad [L] \cap [Q] = 0.
$$
Here $H_{*}^{\rm{lf}}(M)$ is locally-finite homology (also known as Borel-Moore homology), which is
Poincare dual to cohomology with compact support.  Since $L$ and $Q$ are transverse,
it follows from $[L] \cap [Q] = 0$ that $L$ and $Q$ intersect an even number of times.
\end{proof}

\section{Existence of a differential from a path chord} \label{s:pathchordexistence}

In this section we will show how to strengthen Proposition~\ref{p:thediff} to  
Theorem~\ref{t:ExistsPCD}.  The non-trivial part here is to prove that the chord $x \in CF^{*}(L;H)$
given by Proposition~\ref{p:thediff} with $\ip{d_{J}x, q_{0}} \not=0$ and $\cA_{H,L}(x) \geq - a$
can actually be taken to be a path chord.  

This proof has four parts.  In Section~\ref{s:nuf} we will introduce a filtration on $CF^{*}(L;H)$ given
by the Liouville class $\th|_{Q}$ in $Q$.  In Section~\ref{s:CBAB} we will use this filtration to find an 
upper bound for the cotangent bundle action 
$\cA^{T^{*}Q}_{H,L}(x)$
from \eqref{e:CTBA} for chords satisfying the conclusion of Proposition~\ref{p:thediff}.
In Section~\ref{s:nearchord} we will use the assumption that $Q$ has a metric $g$ with non-positive curvature,
along with the bound on the cotangent bundle action and the index relation in Proposition~\ref{p:indexR}, to prove
$x$ can be taken to be a near path chord.  Finally in Section~\ref{s:indexnear} we prove 
the required index relation of Proposition~\ref{p:indexR}.  Section \ref{s:nearchord} is the only
place in the paper where the assumption that $Q$ admits a metric with non-positive curvature is used.

\subsection{The Liouville-filtration}\label{s:nuf}

Recall our fixed admissible Lagrangian $L$ from Lemma~\ref{l:L}
and our Liouville $1$-form $\th$ on $M^{2n}$
such that $\th|_{L} = 0$.  

\subsubsection{The Liouville-filtration}\label{s:nu}

For an admissible Hamiltonian 
$H \in \cH_{\cN}^{Q}$
and a chord $x \in \cC_{H}(L)$ with $x(t) = (q(t), p(t))$ in coordinates for $T^{*}Q$,
denote the integral along the corresponding geodesic of the (negative of the) Liouville class
$\th|_{Q}$ of $Q$ by
\begin{equation}\label{e:nu}
	\nu(x) := -\int_{0}^{1} q^{*}(\th|_{Q})
\end{equation}
For our purposes it will be helpful to have the following alternative description of $\nu$.
Fix a neighborhood $\cN(L)$ of $L$ such that $\cN(L) \cap \cN$ deformation
retracts onto $L \cap Q$ and let
$\cN_{L \cup Q} = \cN(L) \cup \cN$, where $\cN$ is the Weinstein neighborhood of $Q$ in Lemma~\ref{l:L}.
By shrinking $\cN_{L \cup Q}$ slightly we may assume it has
a smooth boundary.
\begin{lem}\label{l:NuEta}
	There is closed $1$-form $\eta$ defined on $\cN_{L \cup Q} \subset (M, d\th)$
	such that $\eta|_{L} = 0$,
	\begin{equation}\label{e:nu=eta}
		\nu(x) = -\int_{0}^{1}x^{*}\eta	\quad\mbox{for chords $x \in \cC_{H}(L)$}\,,
	\end{equation}
	and $\eta = \th - \Psi_{*}\l_{Q}$ in $\cN$
	where $\Psi$ is from \eqref{e:FWN} and $\l_{Q}$ is the canonical $1$-form.
\end{lem}
\begin{proof}
We define $\eta:= \th - \Psi_{*}\l_{Q}$ inside $\cN$ and we need to extend $\eta$ over $\cN(L)$.
Since $L$ is a union of cotangent fibres inside $\cN$ by \eqref{e:NcL} and since $\th|_{L} = 0$, we have $\eta|_{L \cap \cN} = 0$.
Since $\eta|_{\cN(L) \cap \cN}$
is a closed $1$-form on a disjoint union of contractible domains, 
it is exact $d\psi = \eta$ for a function $\psi:\cN(L) \cap \cN \to \R$.
We can assume that $\psi = 0$ on $L \cap \cN$ since
$\eta|_{L \cap \cN} = 0$, so by using bump functions 
we can extend $\psi$ to a compactly supported function on $\cN(L)$ that vanishes
on $L$ and $d\psi$ agrees with $\eta$ on $\cN(L) \cap \cN$.  So $d\psi$ lets us extend
$\eta$ to $\cN_{L \cup Q}$ as desired.

For paths $x:[0,1] \to \cN_{L \cup Q}$ that start and end on $L$, the integral
$\int_{0}^{1} x^{*}\eta$ only depends on the homology class $[x] \in H_{1}(\cN_{L \cup Q}, L)$
since $\eta$ is closed and $\eta|_{L} = 0$.  Therefore \eqref{e:nu=eta} follows 
since a chord $x \in \cC_{H}(L)$
is homologous to its projection $q$ to $Q$ and $\eta|_{Q} = \th|_{Q}$.
\end{proof}

The cotangent bundle action $\cA_{H,L}^{T^{*}Q}(x)$ is equal to $\int_{0}^{1}H(x(t)) dt - \int_{0}^{1}x^{*}\l_{Q}$ 
so by Lemma~\ref{l:NuEta} we have the relation
\begin{equation}\label{e:ActionCtbaNu}
	\cA_{H,L}(x) = \cA_{H,L}^{T^{*}Q}(x) + \nu(x).
\end{equation}

Consider the following class of admissible almost complex structures $J$.
\begin{defn}\label{d:CylJinN}
	For a Weinstein neighborhood $\cN$ of $Q$ and a metric $g$ on $Q$, 
	denote by $\cJ_{\cyl,g}(\cN) \subset \cJ_{\th}(M)$ the
	admissible almost complex structures on $(M, d\th)$ from Definition~\ref{d:cJ}
	that also satisfy the following additional condition.  Near $\d\cN$ they are time-independent and
	agree with the push-forward of some almost complex structure $J$
	on $T^{*}Q \setminus Q$ that is contact type, meaning $\l_{Q} \circ J = dr$ where $r:T^{*}Q \to \R$
	is $r(q,p) = \abs{p}_{g}$.
\end{defn}
We will now show for these almost complex structures that $\nu$ defines a filtration on 
the complex
$(CF^{*}_{(-\infty, 0]}(L;H), d_{J})$ and detects when a differential leaves the Weinstein neighborhood $\cN$ of $Q$.  We have
chosen the minus sign in the definition of $\nu$ so that the differential does not decrease the $\nu$ value just like the action functional $\cA_{H,L}$.

\begin{lem}\label{l:nufilter}
	For $J \in \cJ_{\cyl,g}(\cN)$ and a Hamiltonian $H \in \cH_{\cN}^{Q}$,
	let $u \in \cM(x_{-}, x_{+}; L, H, J)$ solve \eqref{e:FEL}
	where $x_{-}$ and $x_{+}$ are chords contained in $\cN$, then
	$$
		\nu(x_{-}) \geq \nu(x_{+})
	$$
	with equality if and only if $u$ is contained in $\cN$.  Likewise for
	$u \in \cM(x_{-}, x_{+}; L, H^{s}, J^{s})$ for
	a homotopy $H^{s} \in \cH_{\cN}^Q$ between $H^{\pm} \in \cH_{\cN}^Q$
	and $J^{s} \in \cJ_{\cyl,g}(\cN)$. 
\end{lem}
\begin{proof}
	Since $\eta$ is closed and $\eta|_{L} = 0$, for $u \in \cM(x_{-}, x_{+}; L, H, J)$ we have
	$$
		0 = \int_{u^{-1}(\cN)} u^{*}d\eta = \int_{x_{+}} \eta - \int_{x_{-}} \eta + \int_{u^{-1}(\d\cN)} u^{*}\eta
	$$
	so it suffices to prove
	\begin{equation}\label{e:sufficeEta}
		\int_{u^{-1}(\d\cN)} u^{*}\eta \leq 0
	\end{equation}
	with equality if and only if $u$ is contained in $\cN$.  The if direction is immediate
	since if $u$ is contained in $\cN$ then $u^{-1}(\d\cN) = \emptyset$.
	
	For the other direction we will argue as in \cite[Lemma 7.2]{AbS10}.
	Suppose $u$ leaves $\cN$, so let 
	$S = u^{-1}(M \backslash \cN)$ 
	and write $\d S = \d_{l}S \cup \d_{n}S$ where $u(\d_{l}S) \subset L$ and 
	$u(\d_{n}S) \subset \d(M \backslash \cN)$.  
	If $\zeta$ is a vector tangent to $\d_{n}S$ with a positive
	orientation, then
	$j\zeta$ points inwards in $S$ and hence $dr(du(j\zeta)) \geq 0$ where $r: T^{*}Q \to \R$ is given by 
	$r(q,p) = \abs{p}_{g}$.
	
	Since $J \in \cJ_{\cyl,g}(\cN)$, by definition $\l_{Q} \circ J = dr$ near $\d\cN$,
	and $u|_{S}$ is $J$-holomorphic we have
	$$
		\l_{Q}(du(\zeta)) = -(\l_{Q} \circ J)(du(j\zeta)) = -dr(du(j\zeta)) \leq 0
	$$
	and therefore
	$$
		\int_{\d_{n}S} u^{*}\l_{Q} \leq 0.
	$$
	Furthermore since $u|_{S}$ is $J$-holomorphic, $\th|_{L} = 0$, and $\th = \l_{Q} + \eta$ in $\cN$
	we get that
	\begin{equation}\label{e:NuEnergyE}
		0 \leq E(u|_{S}) = \int_{S} \norm{\d_{s}u}^{2}_{J}\, dsdt 
		= \int_{S} u^{*}(d\th)  
		= \int_{\d_{n}S} u^{*}\l_{Q} + \int_{\d_{n}S} u^{*}\eta
		\leq \int_{\d_{n}S} u^{*}\eta.
	\end{equation}
	This equation \eqref{e:NuEnergyE} proves \eqref{e:sufficeEta} since the domains of
	integration have opposite orientations.  By \eqref{e:NuEnergyE} equality in 
	\eqref{e:sufficeEta} implies that $E(u|_{S}) = 0$, i.e.\ that $u|_{S}$ is constant which is impossible
	if $u$ leaves $\cN$.
\end{proof}

\subsubsection{The associated graded complex}\label{s:associated}
 
Let $\gamma$ be the homotopy type of a path in $Q$ that 
starts and ends at $q_{i}, q_{j} \in Q \cap L$, where we will also assume $\gamma$ is non-trivial
if $q_{i} = q_{j}$.  For a Hamiltonian $H \in \cH_{\cN}^{Q}$, define
$$
	CF^{*}_{\nu,\gamma}(L;H) = \Z/2\, \ip{x \in \cC_{H}(L): x \subset \cN \mbox{ and } [\pi(x)] = \gamma}
$$
to be the $\Z/2$ vector space spanned by chords whose geodesic in $Q$ represents $\gamma$
where here $\pi: \cN \to Q$ is the cotangent projection. 
Let 
$$\cM_{1}^{\nu}(y, x; L, H, J) \subset \cM_{1}(y, x; L, H, J)$$
be the Floer trajectories from \eqref{e:diffL} that are contained in the Weinstein neighborhood $\cN$ 
and define $d_{J}^{\nu}: CF^{*}_{\nu,\gamma}(L;H) \to CF^{*+1}_{\nu, \gamma}(L; H)$
by
$$
	d_{J}^{\nu}x = \sum_{y} \#_{\Z/2} (\cM_{1}^{\nu}(y, x; L, H, J)/\R)\, y.
$$
Since all the chords generating $CF^{*}_{\nu, \gamma}(L; H)$ have the same $\nu$-value,
by Lemma~\ref{l:nufilter} the standard gluing and compactness results show that
$d_{J}^{\nu}$ is a differential if $J \in \cJ_{\cyl,g}(\cN)$ is regular with respect to $H$.
We will denote the resulting homology groups by
$$
	HF^{*}_{\nu, \gamma}(L;H) = H^{*}(CF^{*}_{\nu, \gamma}(L; H), d_{J}^{\nu}).
$$
Since we are no longer restricting ourselves to a certain action window, Lemma~\ref{l:nufilter}
also shows that continuation maps give isomorphisms 
$HF^{*}_{\nu, \gamma}(L;H) \cong HF^{*}_{\nu, \gamma}(L;K)$ between different $H, K \in \cH_{\cN}^{Q}$.
In particular since $CF^{*}_{\nu, \gamma}(L;K) = 0$ when the slope of $K$ is less than
the length of any geodesic in the homotopy class $\gamma$, it follows that
\begin{equation}\label{e:ass0}
	HF^{*}_{\nu, \gamma}(L; H) = 0
\end{equation}
for any $H \in \cH_{\cN}^{Q}$.


\subsection{Bounding the cotangent bundle action}\label{s:CBAB}

In this subsection we will use the Liouville-filtration to prove Proposition~\ref{p:CTBAUB}, which gives a bound on the cotangent bundle action $\cA_{H,L}^{T^{*}Q}(x)$ from \eqref{e:CTBA} in terms of the action $\cA_{H,L}(x)$ for chords $x \in \cC_{H}(L)$ connected to 
$q_{0}$ by a differential.

\subsubsection{Finitely many homology classes}

Fix an admissible almost complex structure $J \in \cJ_{\th}(M)$ from Definition~\ref{d:cJ} that is time independent
outside of the Weinstein neighborhood $\cN$ from Lemma~\ref{l:L}.
\begin{lem} \label{l:nubound}
	For $A > 0$, there is an $\eps_{0} > 0$ sufficiently small 
	so that there are only a finite number $N_{A, J}$ of homology classes $\z \in H_{1}(\cN_{L \cup Q}, L)$ satisfying the following property:
	There is a Hamiltonian $H \in \cH_{\cN}^{Q}$,
              a chord $x \in \cC_{H}(L)$, and a $J' \in \cJ_{\th}$ satisfying:
	\begin{enumerate}
	\item[(i)]
	$[x] = \zeta \in H_{1}(\cN_{L \cup Q}, L)$ and $\cA_{H,L}(x) \geq -A$.
	\item[(ii)] the moduli space $\cM(q_{0}, x; L, H, J')$ from \eqref{e:MLFC} 
	is non-empty.
	\item[(iii)] $J'$ is within 
	$\eps_{0}$ of $J$ in the uniform
	$C^{\infty}$-metric outside $\cN$.
	\end{enumerate}
	In particular there is a constant $C_{A,J} \geq 0$ such that $-\nu(x) \leq C_{A,J}$ for any such chord $x$.
\end{lem}
\begin{proof}
	Observe that since $\cA_{H,L}(q_{0}) = -\e_{H}$, the bound on $\cA_{H,L}(x)$ is equivalent
	to the uniform bound $E(u) \leq A - \e_{H}$ for $u \in \cM(q_{0}, x; L, H, J')$ by the a priori energy bound
	\eqref{e:actionbound}.
	
	By contradiction assume there is an infinite number of homology classes, then we have a sequence
	$H_{k} \in \cH_{\cN}^{Q}$ and maps $u_{k} \in \cM(q_{0}, x_{k}; L, H_{k}, J_{k})$ with energy bounded by 
	$E(u_{k}) \leq A$
	such that the homology classes $[x_{k}] \in H_{1}(\cN_{L \cup Q}, L)$ are pairwise distinct
	and outside of $\cN$ we have $C^{\infty}$-convergence $J_{k} \to J$.
	We may assume each $u_{k}$ leaves the neighborhood $\cN_{L \cup Q}$, since if a
	$u_{k}$ does not leave the neighborhood then it gives the relation 
	$[x_{k}] = [q_{0}] \in H_{1}(\cN_{L \cup Q}, L)$.
	
	For $\delta \geq 0$, let $S^{\delta}_{k} = u_{k}^{-1}(M \backslash \cN_{L \cup Q}^{\delta})$
	where $\cN_{L \cup Q}^{\delta}$ are those points in $M \backslash \cN$ within $\delta$
	of a point in $\cN_{L \cup Q}$ in terms of the metric induced by $\w$ and $J$.
	By Fish's compactness result \cite[Theorem A]{Fi11} we know that for any
	$\e > 0$ there is a $\delta \in [0, \e)$ and a subsequence 
	of the curves $u_{k}|_{S^{\delta}_{k}}$ that Gromov converges to a $J$-holomorphic
	map
	$$
		u_{\infty}: S_{\infty}^{\delta} \to M \backslash \cN_{L \cup Q}^{\delta}.
	$$
	It follows from the definition of Gromov convergence that for sufficiently large $k$ 
	in the subsequence that
	$$
		[u_{k}(\d S^{\delta}_{k})] = [u_{\infty}(\d S^{\delta}_{\infty})] \in H_{1}(\cN_{L \cup Q})
	$$
	and in particular the subsequence
	$[u_{k}(\d S^{\delta}_{k})]$ in $H_{1}(\cN_{L \cup Q}, L)$ is eventually constant.  However since
	the maps $u_{k}|_{u_{k}^{-1}(\cN_{L \cup Q}^{\delta})}$ show that
	$$
		[x_{k}] = [u_{k}(\d S^{\delta}_{k})] \in H_{1}(\cN_{L \cup Q}, L)
	$$
	this contradicts the fact that the $[x_{k}]$ classes were distinct.
	
	Once there is a bound $N_{A,J}$ on the number of homology classes, the bound
	on the Liouville-filtration comes for free since $\nu$ only depends on the
	homology class.
\end{proof}

\subsubsection{A bound on the cotangent bundle action}

Using the bound on the Liouville-filtration from Lemma~\ref{l:nubound}, we will now bound the cotangent bundle action.  
Recall $\rho_{H}$ from Definition~\ref{d:cHcN}
is the radius of support of $dH$ for $H \in \cH_{\cN}^{Q}$.
\begin{prop}\label{p:CTBAUB}
	For any $A > 0$ and $J \in \cJ_{\th}$, there is a constant $C_{A,J}^{\cN} \geq 0$ satisfying the following property:
	For any $J' \in \cJ_{\th}$ that is $C^{\infty}$-close to $J$ outside of the Weinstein
	neighborhood $\cN$ of $Q$ and any $H \in \cH_{\cN}^{Q}$,
	if $x \in \cC_{H}(L)$ is a chord such that 
$-A \leq \cA_{H,L}(x)$ and $\cM(q_{0}, x; L, H, J')$ is non-empty,
then
	\begin{equation}\label{e:CTBAUB}
		\cA_{H,L}^{T^{*}Q}(x) \leq  \rho_{H}\,C_{A,J}^{\cN}.
	\end{equation}
\end{prop}

\begin{remark}
	Note that since the bound \eqref{e:CTBAUB} holds for all $H \in \cH_{\cN}^{Q}$,
	we can make the right hand side of \eqref{e:CTBAUB} arbitrary small 
	 by requiring $H$ to be such that $0 < \rho_{H}$ is sufficiently small.
\end{remark}

For this proof we will use a conformal symplectomorphism of $M$ supported in
a Weinstein neighborhood $\cN$ that is given by scaling the cotangent fibers of $Q$ under
the identification $\cN = \{(q,p): \abs{p}_{g} < c\}$ from \eqref{e:FWN}.
For positive numbers $\rho < b < c$, let $\phi_{\rho,b}: [0, c) \to [0, c)$ be a diffeomorphism
where
$$
	\phi_{\rho,b}(r) = \tfrac{b}{\rho}\, r \,\,\mbox{ for $r$ near $[0,\rho]$} \quad\mbox{and}\quad
	\phi_{\rho,b}(r) = r \,\,\mbox{ for $r$ near $c$.}
$$
This diffeomorphism $\phi = \phi_{\rho,b}$ 
defines a diffeomorphism $\Phi = \Phi_{\rho,b}$ of $M$ that is the identity outside
the $\cN$ and inside $\cN$ is given by
$\Phi(q,p) = (q, \phi(\abs{p})p)$ in the cotangent bundle coordinates in $T^{*}Q$.

Consider an admissible
Hamiltonian $H = f_{H}(\abs{p})$ in $\cH_{\cN}^{Q}$ where $\rho_{H} \leq \rho$.
Pushing the Hamiltonian vector field $X_{H}$ for $H$ forward by $\Phi$ results in
$\Phi_{*}X_{H} = X_{H_{\Phi}}$ a Hamiltonian vector field for an
admissible Hamiltonian $H_{\Phi} \in \cH_{\cN}$ given by
\begin{equation}\label{e:PhiH}
	H_{\Phi}(q,p) = f_{H_{\Phi}}(\abs{p}) = \int_{0}^{\phi^{-1}(\abs{p})} \phi'(t) f'_{H}(t) dt
\end{equation}
with $\rho_{H_{\Phi}} < b$.

Since the Lagrangian $L$ intersects $\cN$ along cotangent fibers \eqref{e:NcL}
it follows that $\Phi$ preserves $L$ set-wise.
One can now check that if $x \in \cC_{H}(L)$ is a chord for $H$, then its image
$\Phi(x) \in \cC_{H_{\Phi}}(L)$ is a chord for $H_{\Phi}$ where
\begin{equation}\label{e:PhiAndNu}
	\nu(x) = \nu(\Phi(x)) \quad\mbox{since $x$ and $\Phi(x)$ are homologous in $H_{1}(\cN_{L \cup Q}, L)$.}
\end{equation}
It follows from \eqref{e:PhiH} that for $r$ near $[0, b]$ we have 
$f_{H_{\Phi}}(r) = \tfrac{b}{\rho} f_{H}(\tfrac{\rho}{b} r)$ and hence 
the cotangent bundle
actions \eqref{e:CTBA} of $x$ and $\Phi(x)$ are related by
\begin{equation}\label{e:CTBAandPhi}
	\cA_{H_{\Phi},L}^{T^{*}Q}(\Phi(x)) = \tfrac{b}{\rho} \cA_{H,L}^{T^{*}Q}(x).
\end{equation}
We can now prove Proposition~\ref{p:CTBAUB}.
\begin{proof}[Proof of Proposition~\ref{p:CTBAUB}]
	Consider a diffeomorphism $\Phi$ of $M$ associated to a $\phi_{\rho_{H}, b}$ where $b < c$.
	Observe for chords $x_{\pm} \in \cC_{H}(L)$ that $\Phi$ induces a correspondence between
	elements of the moduli spaces 
	$$
	\cM(x_{-}, x_{+}; L, H, J) \quad\mbox{and}\quad 
	\cM(\Phi(x_{-}), \Phi(x_{+}); L, H_{\Phi}, \Phi_{*}J)
	$$
	from \eqref{e:MLFC}.  In particular if $x \in \cC_{H}(L)$ is a chord such that
	$\cM(q_{0}, x; H, J)$ is non-empty then $\cM(q_{0}, \Phi(x); H_{\Phi}, \Phi_{*}J)$
	is non-empty and hence by the a priori energy estimate \eqref{e:actionbound}
	we know that $\cA_{H_{\Phi}, L}(\Phi(x)) \leq 0$.  Therefore by
	the relations \eqref{e:ActionCtbaNu},
	\eqref{e:PhiAndNu}, and \eqref{e:CTBAandPhi} we have that
	$$
		0 \geq \cA_{H_{\Phi}, L}(\Phi(x)) = 
		\cA_{H_{\Phi}, L}^{T^{*}Q}(\Phi(x)) + \nu(\Phi(x)) = 
		\tfrac{b}{\rho_{H}} \cA_{H,L}^{T^{*}Q}(x)
		+ \nu(x)
	$$
	so therefore $\cA_{H,L}^{T^{*}Q}(x) \leq -\tfrac{\rho_{H}}{b} \nu(x)$ for all $b < c$
	and hence 
	$$\cA_{H,L}^{T^{*}Q}(x) \leq -\tfrac{\rho_{H}}{c} \nu(x).$$
	Since we know $-\nu(x) \leq C_{A,J}$ from Lemma~\ref{l:nubound},
	we are done with $C_{A,J}^{\cN} = \tfrac{C_{A,J}}{c}$.
\end{proof}


\subsection{Existence of a differential from a near path chord}\label{s:nearchord}

We will now turn to the proof of Theorem~\ref{t:ExistsPCD} and let us remind the reader of the near/far
and path/loop dichotomies for chords from Section~\ref{s:cHcN}.
We will first need the following lemma, where $g$ is a metric on $Q$ with nonpositive curvature,
$\cN$ is any Weinstein neighborhood of $Q$, the Hamiltonian
$H \in \cH_{\cN\!, g}^{Q}$ is an admissible Hamiltonian, and $J \in \cJ_{\cyl,g}(\cN)$ is
an admissible almost complex structure from Definition~\ref{d:CylJinN}. 

\begin{lem}\label{l:lowerorderchord}
	If $x_{f} \in \cC_{H}(L)$ is a far chord of $H$ such that $\gen{d_{J}x_{f}, q_{0}} \not= 0$
	in $(CF^{*}(L;H), d_{J})$,
	then there is another chord $x' \in \cC_{H}(L)$ such that
	$\gen{d_{J}x', q_{0}} \not= 0$ with 
	$$
		\cA_{H,L}(x') > \cA_{H,L}(x_{n}) \quad\mbox{and}\quad \nu(x') > \nu(x_{f}) = \nu(x_n)
	$$
	where $x_{n}$ is the corresponding near chord to $x_{f}$.
\end{lem}
\begin{proof}
	Let $\gamma$ be the homotopy type of the path in $Q$ that the chord $x_{f}$ represents.
	Since $g$ is a metric of non-positive curvature, by the 
	Cartan-Hadamard theorem $\gamma$ contains exactly one
	geodesic \cite[Theorem 19.2]{Mi63}.
	Hence $x_{f}$ and the corresponding near version $x_{n}$ are the only
	chords in the associated graded complex $CF^*_{\nu,\gamma}(L;H)$ from Section~\ref{s:associated}.
	Since 
	$\cA_{H,L}(x_{n}) < \cA_{H,L}(x_{f})$ it must be the case that $d^{\nu}_{J} x_{n} = x_{f}$
	in order for $HF_{\nu, \gamma}^{*}(L;H) = 0$, which we know by \eqref{e:ass0}.
	
	Returning to $(CF^{*}(L;H), d_{J})$ we have $\gen{d_{J}x_{n}, x_{f}} \not= 0$.
	Since $\gen{d_{J}x_{f}, q_{0}} \not= 0$, to ensure $(d_{J})^{2}x_{n} = 0$
	there must be another chord
	$x' \not= x_{f}$ such that 
	$$\mbox{$\gen{d_{J}x_{n}, x'} \not= 0$
	and $\gen{d_{J}x', q_{0}} \not= 0$.}
	$$
	That $\gen{d_{J}x_{n}, x'} \not= 0$ implies 
	$\cA_{H,L}(x_{n}) < \cA_{H,L}(x')$ and $\nu(x_{n}) \leq \nu(x')$.
	In fact $\nu(x_{n}) < \nu(x')$ since otherwise by Lemma~\ref{l:nufilter}
	the differential connecting them does not leave $\cN$ and its projection to
	$Q$ provides a homotopy between the corresponding geodesics,
	which would imply $x' = x_{f}$ by the non-positive curvature assumption.
\end{proof}

\begin{proof}[Proof of Theorem~\ref{t:ExistsPCD}]
	Pick $\e > 0$ small enough so that
	$a-\e > e(Q;M)$.
	From Proposition~\ref{p:thediff} we know there is a $K \in \cH^{Q}_{\cN}$
	so that for any $H \geq K$ in $\cH^{Q}_{\cN}$ and any
	$J \in \cJ_{\th}(M)$ that there is a chord $x^{(0)} \in \cC_{H}(L)$ so that
	$\gen{d_{J}x^{(0)}, q_{0}} \not = 0$ with $\cA_{H,L}(x^{(0)}) > -(a-\e)$.
	
	Pick a $J \in \cJ_{\cyl,g}(\cN) \cap \cJ_{\i}(\bB^{2n}_{R})$.	
	Using the constant $N_{a,J}$ from Lemma~\ref{l:nubound}, let $H \geq K$ be any
	admissible Hamiltonian so that
	the bound $B_{f_{H}}$ from \eqref{e:CTBAB} satisfies $\abs{B_{f_{H}}} < \tfrac{\e}{2N_{a,J}}$
	and $\rho_{H}$ is small enough so that the bound in Proposition~\ref{p:CTBAUB} satisfies 
	$\abs{\rho_{H}\,C_{a,J}^{\cN}} < \tfrac{\e}{2N_{a,J}}$.
	
	Suppose $x^{(0)} = x_{f}^{(0)} \in  \cC_{H}(L)$ is a far chord, then we have
	$$
		\cA_{H,L}(x_{n}^{(0)}) = \cA_{H,L}(x_{f}^{(0)}) + \cA_{H,L}^{T^{*}Q}(x_{n}^{(0)}) - \cA_{H,L}^{T^{*}Q}(x_{f}^{(0)})
		> -(a-\e) - \tfrac{\e}{N_{a,J}}
	$$
	using the bound from \eqref{e:CTBAB} on $\cA_{H,L}^{T^{*}Q}(x_{n}^{(0)})$ and the bound
	from Proposition~\ref{p:CTBAUB} on $\cA_{H,L}^{T^{*}Q}(x_{f}^{(0)})$.  Lemma~\ref{l:lowerorderchord}
	gives us a new chord $x^{(1)} \in \cC_{H}(L)$ with $\gen{d_{J}x^{(1)}, q_{0}} \not= 0$ such that
	$$
		\cA_{H,L}(x^{(1)}) > \cA_{H,L}(x_{n}^{(0)}) > -(a-\e) - \tfrac{\e}{N_{a,J}}
		\quad\mbox{and}\quad \nu(x^{(1)}) > \nu(x_{n}^{(0)}).
	$$
	If $x^{(1)} = x_{f}^{(1)}$ is a far chord, then since $\cA_{H,L}(x^{(1)}_{f}) > 
	-(a-\e)$
	we can repeat this argument to get a chord $x^{(2)} \in \cC_{H}(L)$ with $\gen{d_{J}x^{(2)}, q_{0}} \not= 0$ 
	such that
	$$
		\cA_{H,L}(x^{(2)}) > \cA_{H,L}(x_{n}^{(1)}) > -(a-\e) - \tfrac{2\e}{N_{a,J}}
		\quad\mbox{and}\quad \nu(x^{(2)}) > \nu(x_{n}^{(1)}) > \nu(x_{n}^{(0)}).
	$$
	There are only $N_{a,J}$ possible values for $\nu$ on such chords by Lemma~\ref{l:nubound},
	so this process must terminate after at most $N_{a,J}$ steps
	with a near chord $x_{n} \in \cC_{H}(L)$ such that 
	$\gen{d_{J}x_{n}, q_{0}} \not= 0$ and with action
	$\cA_{H,L}(x_{n}) > -a$.
	
	Since $\abs{q_{0}}_{\Mas} = 0$ in $\Z/2$, for degree reasons it must be the
	case that $\abs{x_{n}}_{\Mas} = 1$ in $\Z/2$. 
	If $x_{n}$ was a near loop chord, then it follows from 
	Corollary~\ref{c:indexR} below that the Morse index of the underlying
	geodesic $q$ in $Q$ satisfies $m_{\O}(q) = 1$ in $\Z/2$.
	When $g$ is a metric of non-positive curvature, this is impossible
	since every geodesic $q: [0,1] \to Q$ has Morse index $m_{\O}(q) = 0$,
	see for instance \cite[Section 19]{Mi63}.  
	Therefore $x_{n}$ is a path chord.
\end{proof}

\begin{remark}
	The index argument for ruling out near loop chords does not apply
	to far loop chords since for far loop chords
	\eqref{e:indexRI} is shifted by $+1$ to
	$\abs{(x_f,v)}_{\Mas} = -m_{\O}(q) - \mu_{Q}([v]) + 1$.
\end{remark}


\subsection{The index of a near loop chord}\label{s:indexnear}

Let $g$ be a metric on a closed oriented manifold $Q$ and on the cotangent bundle $T^{*}Q$
let $H$ be a Hamiltonian of the form
\begin{equation}\label{e:geoH}
	H(q,p) = f_{H}(\abs{p}_{g}) \quad \mbox{such that $f_{H}'(r) > 0$ and $f_{H}''(r) > 0$ when $r>0$.} 
\end{equation}
The (not necessarily contractible) Hamiltonian chords $x = (q,p) \in \cC_{H}^{*}(T^{*}_{q_{0}}Q)$ on 
$\{\abs{p}_{g} = r\}$ 
are in one-to-one correspondence with geodesic paths $q: [0,1] \to Q$ starting and ending at $q_{0}$ 
with speed $\abs{\dot{q}}_{g} = f_{H}'(r)$.  

Suppose $Q^{n} \subset (M^{2n}, \w)$ is a closed oriented Lagrangian
with a Weinstein neighborhood
$\cN' \subset M$ symplectomorphic to $\{\abs{p}_{g} < c'\} \subset T^{*}Q$.
Let $H: M \to \R$ is a Hamiltonian of the form \eqref{e:geoH} in $\cN'$
and let $L \subset M$ be a simply connected
Lagrangian such that a connected component of $\cN' \cap L$ is identified
with $\{(q_{0}, p) : \abs{p}_{g} < c'\} \subset T^{*}_{q_{0}}Q$.  The main goal of this subsection is to prove the following proposition.

\begin{prop}\label{p:indexR}
Let $x = (q,p)\in \cC_{H}(L)$ is a non-degenerate contractible chord contained in $\cN$
with $q(0) = q(1) = q_{0}$.  Any capping disk $v$ for $x$ determines an element $[v] \in \pi_{2}(M, Q)$,
one has the relation
\begin{equation}\label{e:indexRI}
	\abs{(x,v)}_{\Mas} = -m_{\O}(q) - \mu_{Q}([v])
\end{equation}
and in particular $\abs{x}_{\Mas} = m_{\O}(q)$ in $\Z/2$.
\end{prop}
Here $\abs{(x,v)}_{\Mas}$ is the Maslov index of the chord, $m_{\O}(x)$ is the Morse index of the underlying
geodesic path in $Q$, and $\mu_{Q}([v])$ is the Maslov index of the element $[v] \in \pi_{2}(M, Q)$. 
The definitions of the various indices are recalled below and $\abs{x}_{\Mas} := \abs{(x,v)}_{\Mas} \in \Z/2$ gives the $\Z/2$ grading to $CF^{*}(L;H_{Q})$.
Proposition~\ref{p:indexR} specializes to the following corollary in the setting of Section~\ref{s:cHcN}.
\begin{cor}\label{c:indexR}
	For a Hamiltonian $H \in \cH_{\cN,g}^{Q}$, if $x \in \cC_{H}(L)$ is a non-degenerate near loop chord
	with corresponding geodesic $q$ in $(Q,g)$, then $\abs{x}_{\Mas} = m_{\O}(q)$ in $\Z/2$.
\end{cor}
\begin{proof}
	Recall from Section~\ref{s:cHcN} that all near chords appear in the region of $\cN$ where $H$ has the
	form \eqref{e:geoH}, where in particular $f_H'' > 0$, and hence Proposition~\ref{p:indexR} applies.
\end{proof}
  

While Proposition~\ref{p:indexR} is most likely well-known to experts, we do not
know of a reference so we will give a proof in Section~\ref{ss:proofVR}.
Before we present the proof though, for clarity and the convenience of the reader we will briefly establish our conventions for various Maslov indices.
As our primitive notion, for a path 
$\Lambda: [a,b] \to \cL_{n}$ in the Lagrangian Grassmanian for $(\R^{2n}, dx\wedge dy)$ and a fixed Lagrangian $V \in \cL_{n}$ we will let $\mu_{\Mas}(\Lambda;V)$ be the \textbf{Maslov index}
as defined by Robbin-Salamon in \cite[Section 2]{RS93}
and we will set $V_{0} = 0 \times \R^{n}$.  The normalization for $\mu_{\Mas}$ is set so
$\mu_{\Mas}(\{e^{2\pi i k t}V_{0}\}_{t\in[0,1]}; V_{0}) = 2k$ for $V_{0} = 0 \times \R$ in $(\R^{2}, dx\wedge dy)$ where $k \in \Z$ is an integer.

\subsubsection{The Maslov class of a Lagrangian}\label{ss:maslov}

The \textbf{Maslov class} of a Lagrangian $Q \subset (M, \w)$ is a homomorphism 
$\mu_{Q}: \pi_{2}(M, Q) \to \Z$, which for a smooth map
$u: (\bD^{2}, \d \bD^{2}) \to (M, Q)$ is defined by
$$
	\mu_{Q}(u) := \mu_{\Mas}(\Lambda_{u}; V_{0}).
$$
For $q(t) = u(e^{2\pi i t})$, the loop $\Lambda_{u}: S^{1} \to \cL_{n}$ is defined by
$$
\Lambda_{u}(t) = \Phi_{u}(t)^{-1}(T^{vert}_{q(t)}T^{*}Q)
$$
where $\Phi_{u}: S^{1} \times \R^{2n} \to q^{*}(TT^{*}Q)$ is a symplectic trivialization
and $T^{vert}T^{*}Q \subset TT^{*}Q$ is the vertical tangent bundle.
The Maslov class has the property that $\mu_{Q}(u) \in 2\Z$ if $Q$ is orientable.

\subsubsection{The Maslov index of a contractible chord with a capping disk}\label{ss:maslovcap}
 
For any Lagrangian $L \subset (M^{2n}, \w)$ and Hamiltonian $H: [0,1] \times M \to \R$,
let $x \in \cC_{H}(L)$ be a contractible non-degenerate Hamiltonian chord
and let $v$ be a capping disk of $x$, i.e.\ \eqref{e:LagCap} a smooth map
$$
	v: \bD^{2} \to M \quad\mbox{such that $v(e^{\pi i t}) = x(t)$ and $v(e^{-\pi i t}) \in L$ for $t\in [0,1]$}.
$$
The \textbf{Maslov index} of the pair $(x,v)$ is defined to be 
$$
	\mu(x,v) := \mu_{\Mas}(\Lambda_{(x,v)}; V_{0})
$$
where the path $\Lambda_{(x,v)}: [-1,1] \to \cL_{n}$ is defined by the concatenation
\begin{equation}\label{e:maslovindexformula}
	\Phi_{(x,v)}(e^{i \pi t})\Lambda_{(x,v)}(t) 
	= \{T_{v(e^{\pi i t})}L\}_{t \in [-1, 0]} \#\{d\vp_{H}^{t}(T_{x(0)}L)\}_{t \in [0,1]}
\end{equation}
where $\Phi_{(x,v)}: \bD^{2} \times \R^{2n} \to v^{*}TM$ is a symplectic trivialization
with $\Phi_{(x,v)}(-1)V_{0} = T_{x(1)}L$.  It follows from the homotopy
invariance of the Malsov index that if two capping disks $v$ and $v'$ of $x$ are homotopic
through capping disks of $x$, then the indices $\mu(x,v) = \mu(x, v')$ are equal.

This is the index used to grade Lagrangian Floer cohomology, more precisely if $x$ is a non-degenerate
contractible chord and $v$ is a capping disk define
\begin{equation}\label{e:LFCG}
	\abs{(x,v)}_{\Mas} = -\mu(x,v) + \tfrac{n}{2} \in \Z.
\end{equation}
When $L$ is orientable, this induces a $\Z/2$-grading $\abs{x}_{\Mas}$ on contractible non-degenerate chords
$x \in \cC_{H}(L)$ by
\begin{equation}\label{e:Lgrading}
	\abs{x}_{\Mas} \equiv \abs{(x,v)}_{\Mas}  \pmod{2} \quad \mbox{for any capping disk $v$.}
\end{equation}
The definition of $\abs{x} \in \Z/2$ is well-defined since if $v_{1}$ and $v_{2}$ are capping disks for the same chord $x$, then $\mu(x,v_{1}) - \mu(x, v_{2}) = \mu_{L}(v_{1}\#\bar{v}_{2}) \in 2\Z$
where $v_{1}\#\bar{v}_{2} \in \pi_{2}(M, L)$ is the result of gluing $v_{1}$ and $v_{2}$
along the chord $x$.

\subsubsection{The Maslov and Morse indices of a chord in a cotangent bundle}\label{ss:maslovcot}

Let $x = (q,p) \in \mathcal{\cC}^{*}_{H}(T_{q_{0}}^{*}Q)$ be any chord for a Hamiltonian
$H: [0,1] \times T^{*}Q \to \R$ on $(T^{*}Q, d\l_{Q})$, then the \textbf{internal Maslov index} of $x$ is defined as
$$
	\mu_{\Int}(x) := \mu_{\Mas}(\Lambda_{x}; V_{0}).
$$
Here $\Lambda_{x}: [0,1] \to \cL_{n}$ is the defined by
$$
	\Lambda_{x}(t) := \Psi_{x}(t)^{-1}(d\vp_{H}^{t} T_{x(0)}^{vert}T^{*}Q)
$$
where $\Psi_{x}: [0,1] \times \R^{2n} \to x^{*}(TT^{*}Q) = q^{*}(TQ \oplus T^{*}Q)$
is a symplectic trivialization such that $\Psi_{x}(t)(V_{0}) = T^{vert}_{x(t)}T^{*}Q = T^{*}_{q(t)}Q$.
Such trivializations always exist and $\mu_{\Int}(x)$ is independent of the choice of $\Psi_{x}$
see for instance \cite[Lemmas 1.2 and 1.3]{AS06}. 

Assume that $H: [0,1] \times T^{*}Q \to \R$ has the form \eqref{e:geoH}, then chords
$x \in \mathcal{C}_{H}^{*}(T_{q_{0}}^{*}Q)$ correspond to geodesic paths, namely
critical points of the functional
$$
	\mathcal{E}_{g}(q) = \int_{0}^{1} \abs{\dot{q}(t)}^{2}_{g} dt
$$
on the space $\O_{M}(q_{0}, q_{0})$ of paths in $Q$ with boundary conditions $q(0) = q(1) = q_{0}$.  Associated
to a geodesic path $q$ is its \textbf{Morse index} $m_{\O}(q)$, which is the number of negative
eigenvalues of the Hessian of $\mathcal{E}_{g}$ at $q$ counted with multiplicity or equivalently the number of conjugate points along the geodesic $q$.  See \cite[Part 3]{Mi63} for details.

For non-degenerate chords in cotangent bundles where $H$ has the form \eqref{e:geoH}
Duistermaat \cite{Du76}, see also \cite[Proposition 6.38]{RS95}, showed that
the Morse index and the internal Maslov index are related as follows.
\begin{prop}\label{p:m=m}
	If $x = (q,p) \in \cC_{H}^{*}(T^{*}_{q_{0}}Q)$ in $(T^{*}Q, d\th)$ is a non-degenerate chord
	for a Hamiltonian $H$ with the form \eqref{e:geoH}, then
	$$
		\mu_{\Int}(x) = m_{\O}(q) + \tfrac{n}{2}.
	$$
\end{prop}
Note that there is a sign discrepancy between \cite[Proposition 6.38]{RS95} and Proposition~\ref{p:m=m} since \cite{RS95} use the symplectic form $-d\l_{Q} = dq\wedge dp$ on $T^{*}Q$.

\subsubsection{Proof of Proposition~\ref{p:indexR}}\label{ss:proofVR}

The proof of Proposition~\ref{p:indexR} reduces to proving
\eqref{e:VR} and this is the direct analogue of an identity for Hamiltonian orbits,
which in a special case was proved by Viterbo \cite[Theorem 3.1]{Vi90a}
and the general case is in \cite[Proposition 4.3]{KS10}.

In the setting of Proposition~\ref{p:indexR} we have a
contractible chord $x = (q,p) \in \cC_{H}(L)$ contained in $\cN'$ whose corresponding geodesic 
$q$ represents a based loop at $q_{0}$.
Any capping disk $v$ of $x$ determines an element $[v] \in \pi_{2}(M, Q)$
because $L$ is simply connected so without loss of generality we can assume the boundary
of $v$ is contained in the Weinstein neighborhood $\cN'$ of $Q$.

\begin{proof}[Proof of Proposition~\ref{p:indexR}]	
	By \eqref{e:LFCG} and Proposition~\ref{p:m=m}, it suffices to prove
	\begin{equation}\label{e:VR}
		\mu(x,v) = \mu_{\Int}(x) + \mu_{Q}([v]) 
	\end{equation}
	since then $\abs{(x,v)}_{\Mas} = -\mu(x,v) + \tfrac{n}{2} = -\mu_{\Int}(x) - \mu_{Q}([v]) + \tfrac{n}{2}
		= -m_{\O}(q) - \mu_{Q}([v])$.
		Furthermore $\mu_{Q}([v]) \in 2\Z$ since $Q$ is orientable, so
		it follows that $\abs{x}_{\Mas} = m_{\O}(q)$ in $\Z/2$.  It remains to prove \eqref{e:VR}.

              The definition of $\Lambda_{(x,v)}(t)$ in (\ref{e:maslovindexformula})
              tells us that $\Lambda_{(x,v)}(t)$ is a concatenation of two paths.
	By multiplying by $\Psi_{x}(t)\Psi_{x}^{-1}(t)$ we can see that
	second path of $\Lambda_{(x,v)}(t)$ is homotopic to the concatenation 
	\begin{equation}\label{e:indexHo}
	\{\Phi_{(x,v)}^{-1}(e^{i\pi t})\Psi_{x}(t)\Psi_{x}^{-1}(0)T_{x(0)}L\}_{t \in [0,1]}\,\, \#\,\,
		\{\Phi_{(x,v)}^{-1}(-1)\Psi_{x}(1)\Psi_{x}^{-1}(t)d\vp_{H}^{t}(T_{x(0)}L)\}_{t\in [0,1]}.
	\end{equation}
              By naturality of the Maslov index in the sense that
              $\mu_{\Mas}(\Lambda; V_{0}) =
              \mu_{\Mas}(A\Lambda; AV_{0})$ for a symplectic matrix $A \in Sp(\R^{2n})$,
              for the second path in \eqref{e:indexHo} we have 
	$$
	\mu_{\Mas}\Big(\Phi_{(x,v)}^{-1}\Psi_{x}(1)\Psi_{x}^{-1}(t)d\vp_{H}^{t}(T_{x(0)}L); V_{0}\Big)
	=
	\mu_{\Mas}\Big(\Psi_{x}^{-1}(t)d\vp_{H}^{t}(T_{x(0)}L); V_{0}\Big) = \mu_{\Int}(x)
	$$
	and the concatenation of the first path in $\Lambda_{(x,v)}(t)$ and the first path 
	in \eqref{e:indexHo} gives
	$$
		\mu_{\Mas}\Big(\{\Phi_{(x,v)}^{-1}(e^{i\pi t})T^{vert}_{v(e^{i\pi t})}T^{*}Q\}_{t \in [-1,1]}; V_{0}\Big)
		= \mu_{Q}([v]).
	$$
	Using the previous Maslov index calculations and the fact that $\mu_{\Mas}$ is additive under concatenation,
	we have
	\begin{align*}
		\mu(x,v) &= 
		\mu_{\Mas}\Big(\{\Phi_{(x,v)}^{-1}(e^{i\pi t})T^{vert}_{v(e^{i\pi t})}T^{*}Q\}_{t \in [-1,1]}, V_{0}\Big)
		+ \mu_{\Mas}\Big(\Psi_{x}^{-1}(t)d\vp_{H}^{t}(T_{x(0)}L); V_{0}\Big)\\
		&= \mu_{Q}([v]) + \mu_{\Int}(x).
	\end{align*}
	using that $\mu_{\Mas}$ is additive under concatenation.
\end{proof}

\subsubsection{Conley-Zehnder index conventions}

Since in the next section we will reference the Conley-Zehnder index, let us take a second
to recall the definition.
Given a symplectic matrix $A \in Sp(2n)$ for $(\R^{2n}, \w_{0} = dx \wedge dy)$ the graph
$gr(A)$ is a Lagrangian subspace in $(\R^{2n} \times \R^{2n}, -\w_{0} \oplus \w_{0})$.
One defines the \textbf{Conley-Zehnder} index of a path $A: [a,b] \to Sp(2n)$
of symplectic matrices to be the Maslov index 
$$\mu_{\CZ}(A) = \mu_{\Mas}(gr(A); \Delta)$$
of the path of Lagrangians $gr(A)$ with respect to the diagonal $\Delta \subset \R^{2n} \times \R^{2n}$.
These conventions are such that $\mu_{\CZ}(\{e^{2\pi i k t}\}_{t \in [0,1]}) = 2k$ for $k \in \Z$ in 
$(\R^{2}, dx\wedge dy)$.



\section{The comparison and energy-capacity inequalities}\label{s:EAILP}

The goal of this section is to prove Theorem~\ref{t:CapIneq}.
We will begin with a brief summary of Hamiltonian Floer cohomology, if only to establish conventions and notations, and then we will give the definition of the Floer--Hofer--Wysocki capacity.

\subsection{The Hamiltonian Floer--Hofer--Wysocki capacity}\label{s:DefFHC}

\subsubsection{Hamiltonian Floer cohomology}

Hamiltonian Floer cohomology on a Liouville manifold $(M^{2n}, d\th)$ \cite{Fl89, FH94, FHS95} is analogous to Lagrangian Floer cohomology in Section~\ref{ss:LFC}
except now one considers $1$-periodic orbits instead of chords.

Given a Hamiltonian $H: S^{1} \times M \to \R$, let
\begin{equation}\label{e:orbits}
        \cO_{H} = \{ x: S^{1} \to M \mid \tfrac{\d}{\d t} x(t) = X_{H_{t}}(x(t)) \quad\mbox{and}\quad
        [x] = 1 \in \pi_{1}(M)\}
\end{equation}
denote the set of contractible Hamiltonian orbits.
An orbit $x \in \cO_{H}$ is non-degenerate if 
$d\vp_{H}^1 : TM_{x(0)} \to TM_{x(0)}$ has no eigenvalue equal to $1$.
A \textbf{capping disk} $v$ of an orbit $x \in \cO_{H}$ is a map
\begin{equation}\label{e:OrbitCap}
        v: \bD^{2} \to M \quad\mbox{such that $v(e^{2\pi i t}) = x(t)$ for $t\in \R / \Z$}
\end{equation}
with which one can build a symplectic trivialization of $x^{*}TM$ 
and turn $d(\vp_{H}^{t})_{x(0)}$ into
a path of symplectic matrices $A_{(x,v)}: [0,1] \to Sp(2n)$ that starts at $\ide$.  
One defines the index of a non-degenerate orbit $x$ with a capping disk $v$ to be
$$
\abs{(x,v)}_{\CZ} = n - \mu_{CZ}(A_{(x,v)}) \in \Z
$$
where $\abs{\cdot}_{\CZ}$ is normalized so that for a $C^{2}$-small Morse function 
$f$ with a critical point $x$ and constant capping disk $v$ we have $\abs{(x,v)}_{\CZ} = \mbox{Morse}_{f}(x)$.
This induces a well-defined $\Z/2$-grading
$$
\abs{x}_{\CZ} := \abs{(x,v)}_{\CZ} \quad\mbox{in $\Z/2$}
$$
since $\abs{(x,v_{1})}_{\CZ} - \abs{(x,v_{2})}_{\CZ} = -2c_{1}(v_{1}\#\overline{v}_{2})$.

\begin{defn}
An admissible Hamiltonian $H \in \cH$ as in Definition~\ref{d:cH} is \textbf{non-degenerate}
if all orbits $x \in \cO_{H}$ with action $\cA_{H}(x) < M_{H}$ are non-degenerate.
\end{defn}
Here the action functional $\cA_{H}: \cO_{H} \to \R$ is given by
\begin{equation}\label{e:HamiltonianActionFunctional}
        \cA_{H}(x) = \int_{0}^{1} H(t, x(t)) dt - \int_{0}^{1} x^{*}\th.
\end{equation}

For non-degenerate orbits $x_{\pm} \in \cO_{H}$
and admissible almost complex structure $J \in \cJ_{\th}(M)$, the moduli space
$\cM(x_{-}, x_{+}; H, J)$
is the set of finite energy solutions $u= u(s,t): \R \times S^{1} \to M$
\begin{equation}\label{e:FE}
	\d_{s}u + J_{t}(u)(\d_{t}u - X_{H_{t}}(u)) = 0
\end{equation}
with asymptotic convergence 
$\lim_{s \to \pm \infty} u(s,\cdot) = x_{\pm}(\cdot)$.
The energy of a solution to \eqref{e:FE} is
\begin{equation}\label{e:EnergyH}
	E(u) := \int_{\R \times S^{1}} \norm{\d_{s}u}^{2}_{J} ds\,dt 
	\quad\mbox{where $\norm{\d_{s}u}^{2}_{J} = d\th(\d_{s}u, J_{t}(u)\d_{s}u)$}
\end{equation}
and there is the standard a priori energy bound
\begin{equation}\label{e:hamiltonianactionbound}
	0 \leq E(u) = \cA_{H}(x_{-}) - \cA_{H}(x_{+})
\end{equation}
for $u \in \cM(x_{-}, x_{+}; H, J)$.
For non-degenerate $H$ and generic admissible $J$ the moduli space $\cM(x_{-}, x_{+}; H, J)$ is a smooth manifold and the dimension near a solution $u \in \cM(x_{-}, x_{+}; H, J)$ is determined by 
\begin{equation}\label{e:dimMod}
\dim_{u} \cM(x_{-}, x_{+}; H, J) = 
\abs{(x_{-},v)}_{\CZ} - \abs{(x_{+},v \# u)}_{\CZ}
\end{equation}
where $v$ is any capping disk for the orbit $x_{-}$.
We will denote by
$\cM_1(x_{-}, x_{+}; H, J)$ the union of the $1$-dimensional connected components of 
$\cM(x_{-}, x_{+}; H, J)$.  Translation in the domain gives an $\R$-action to the moduli space 
$\cM(x_{-}, x_{+}; H, J)$ and $\cM_1(x_{-}, x_{+}; H, J)/\R$ is a compact $0$-dimensional manifold.

The vector space over $\Z/2$ generated by orbits $x \in \cO_{H}$ with action in 
the window $(a,b]$ is denoted by
$$
CF_{(a,b]}^{*}(H) 
$$
and it is $\Z/2$-graded if all the orbits are non-degenerate.
Analogously to the Lagrangian case the $\Z/2$-linear map
\begin{equation}\label{e:HamiltoniandiffL}
	d_{J}: CF^{*}_{(a,b]}(H) \to CF^{*+1}_{(a,b]}(H)	
\end{equation}
given by counting isolated positive gradient trajectories
$$
	d_{J}x = \sum_{y} 
	(\#_{\Z_{2}} \cM_1(y, x; H, J)/\R) \, y
$$
defines a differential, where the sum is over orbits $y \in \cO_{H}$ with action in the window $(a, b]$.
Hamiltonian Floer cohomology 
$$HF^{*}_{(a,b]}(H) = H^{*}(CF^{*}_{(a,b]}(H),d_{J})$$ is the homology
of this chain complex. 

The continuation maps and action window maps for Hamiltonian Floer cohomology are analogous to the Lagrangian case.  In particular, given non-degenerate Hamiltonians $H^{+} \leq H^{-}$ there is a
monotone continuation map
\begin{equation}\label{e:MonoConH}
	\Phi_{H^{+}H^{-}}:  HF^{*}_{(a,b]}(H^{+}) \to HF^{*}_{(a,b]}(H^{-})
\end{equation}
that is independent of the choice of monotone homotopy $(H^{s}, J^{s})$
used to define it.   

\subsubsection{Hamiltonian Floer--Hofer--Wysocki capacity}

\begin{defn}
Let $f: M \to \R$ be an admissible Hamiltonian that is non-degenerate with respect to $M$.
	If the following conditions are satisfied
	\begin{enumerate}
		\item[(i)] every orbit $x \in \cO_{f}$ is a critical point of $f$,
		\item[(ii)]
		the only critical points for $\{f > 0\}$ occur at infinity where $f$ is constant,
		\item[(iii)]
		the regular sublevel set $\{f \leq 0\}$ is a deformation retract of $M$,
		\item[(iv)]
		$f$ is a $C^{2}$-small Morse function on $\{f \leq 0\}$,
	\end{enumerate}
	then we say $f$ is \textbf{adapted} to $M$.
\end{defn}

It follows from \cite[Theorem 7.3]{SZ92} that if $f: M \to \R$ is a Hamiltonian
adapted to $M$ and $f > -a$, then via Morse cohomology one has a chain-level isomorphism 
\begin{equation}\label{e:CLIH}
	H^{*}_{\mbox{\tiny{Morse}}}(M) \cong HF^{*}_{(-a,0]}(f)
\end{equation}
given by mapping critical points $x \in \mbox{Crit}(f)$ with $f(x) < 0$ to the corresponding constant orbit 
$x \in \cO_{f}$.

Just as in the Lagrangian case, for a compact subset $X \subset M$ 
and $a > 0$ we define 
\begin{equation}
	HF^{*}(X, a) := \varinjlim_{H \in \cH^{X}} HF^{*}_{(-a, 0]}(H)
\end{equation}
where monotone continuation maps \eqref{e:MonoConH}
are used for the direct limit over the class of Hamiltonians 
$\cH^{X}$ from \eqref{e:cHW}. 
Similarly there is a natural map
\begin{equation}\label{e:iHFH}
	i_{X}^{a}:  H^{*}(M) \to HF^{*}(X, a)
\end{equation}
given by the isomorphism \eqref{e:CLIH} and 
the inclusion of $HF^{*}_{(-a,0]}(f)$ into the direct limit where $f \in \cH^{X}$ is
adapted to $M$ with $f > -a$.  

We now have the following definition where $\ide_{M} \in H^{0}(M)$ is the fundamental class.
\begin{defn}
	The \textbf{Floer--Hofer--Wysocki capacity} of $X$ is
	\begin{equation}\label{e:cH}
		c^{FHW}(X) = \inf\{ a > 0 : i_{X}^{a}(\ide_{M}) = 0\}
	\end{equation}
	where $c^{FHW}(X) = +\infty$ if 
	$i_{X}^{a}(\ide_{M}) \not= 0$ for all $a > 0$.
\end{defn}

Just like for the Lagrangian case we have the following criterion for when $c^{FHW}(X) < a$,
which follows from the definitions.
\begin{lem}\label{l:CapBoundHam}
	For any finite $a$, the capacity $c^{FHW}(X) \leq a$ if and only if there is
	an $f \in \cH^{X}$ adapted to $M$ and an $H \in \cH^{X}$ so that
	$-a < f \leq H$ and
	$$
		\ide_{M} \in \ker\left(\Phi_{fH} : HF^{*}_{(-a, 0]}(f) \to HF^{*}_{(-a,0]}(H)\right)
	$$
	where $\ide_{M} \in H^{*}(M) \cong HF^{*}_{(-a_{0}, 0]}(f)$ are identified as in \eqref{e:CLIH}.	
\end{lem}


\subsection{Proof of Theorem~\ref{t:CapIneq}}

We will now present proofs of the various inequalities for the Hamiltonian and Lagrangian
Floer--Hofer--Wysocki capacities given in Theorem~\ref{t:CapIneq}.

\subsubsection{Proving part (i): The comparison inequality via a closed-open map}

Theorem~\ref{t:CapIneq}(i) follows directly from
the existence of a closed-open map
\begin{equation}\label{e:COmap}
	\cC\cO: HF^{*}(X, a) \to HF^{*}(L;X,a)
\end{equation}
such that there is a commutative diagram
\begin{equation}\label{e:COcommute}
	\xymatrix{
	H^{*}(M) \ar[r]^{i^{*}} \ar[d]_{i_{X}^{a}} & H^{*}(L) \ar[d]^{i_{L;X}^{a}}\\
	HF^{*}(X,a) \ar[r]^{\cC\cO} &  HF^{*}(L;X,a)
	}
\end{equation}
where $i^{*}: H^{*}(M) \to H^{*}(L)$ is the standard map on cohomology.

\begin{proof}[Proof of Theorem~\ref{t:CapIneq}(i)]
	Since $i^{*}(\ide_{M}) = \ide_{L}$, it follows from \eqref{e:COcommute}
	that $i_{X}^{a}(\ide_{M}) = 0$ implies $i_{L;X}^{a}(\ide_{L}) = 0$.
	By the definitions of the capacities, this proves Theorem~\ref{t:CapIneq}(i).	
\end{proof}

Since closed-open maps have appeared in the literature in \cite{AS06a, AS10, Al08} and many times since,
we will just briefly recall the construction.  
For a given Hamiltonian $H \in \cH^{X}$ that is non-degenerate with respect to $M$ and $L$,
there is a map
\begin{equation}\label{e:CODef}
	\cC\cO: HF^{*}_{(-a,0]}(H) \to HF^{*}_{(-a,0]}(L;H)
\end{equation}
which is Albers' map $\tau$ in \cite[Section 5]{Al08}.
If $y \in \cO_{H}$ is an orbit with action in $(-a, 0]$, then on the chain level \eqref{e:CODef} is defined by
$$
	\cC\cO(y) = \sum_{x} \left(\#_{\Z_{2}} \cM_0^{\cC\cO}(x, y; L, H, J)\right) \, x
$$
where the sum is taken
over chords $x \in \cC_{H}(L)$ with action in $(-a, 0]$.
The moduli space $\cM_0^{\cC\cO}(x, y; L, H, J)$ is the zero dimensional component
of the space of finite energy solutions to $u = u(s,t) : \Si \to M$
\begin{equation}\label{e:COmapE}
	\begin{cases}
	u(\d\Si) \subset L\\
	\d_{s}u + J_{t}(u)(\d_{t}u - X_{H_{t}}(u)) = 0\\
	u(-\infty, \cdot) = x(\cdot),\, u(+\infty, \cdot) = y(\cdot)
	\end{cases}
\end{equation}
where $J = \{J_{t}\}_{t \in S^{1}} \in \cJ_{\th}(M)$ is an admissible almost complex structure
and
$$
	\Si = \R \times [0,1] / \sim \quad\mbox{where}\quad (s,0) \sim (s,1) \mbox{ for $s \geq 0$}
$$
with boundary $\d\Si = \{(s,t): s \leq 0,\, t=0,1\}$.
The energy of a solution to \eqref{e:COmap} is given by
\begin{equation}\label{e:EnergyCO}
	E(u) := \int_{\Si} \norm{\d_{s}u}^{2}_{J} ds\,dt 
	\quad\mbox{where $\norm{\d_{s}u}^{2}_{J} = d\th(\d_{s}u, J_{t}(u)\d_{s}u)$}
\end{equation}
and again we have the a priori energy bound
\begin{equation}\label{e:EBCO}
	0 \leq E(u) = \cA_{H,L}(x) - \cA_{H}(y)
\end{equation}
which is why the map $\cC\cO$ preserves the action filtration.

Standard proofs show that the closed-open map is natural with respect to monotone continuation maps 
in Hamiltonian Floer cohomology \eqref{e:MonoConH} and in Lagrangian Floer cohomology \eqref{e:MonoCon},
and hence the map in \eqref{e:CODef} induces the map \eqref{e:COmap} on the direct limits.
While Albers works in the case where $M$ is closed and $L$ is monotone, his proof generalizes to this setting
since there are no holomorphic disks on $L$ or holomorphic spheres in $M$ and
Lemma~\ref{l:MaxPrin} provides the needed maximum principle.  The commutativity
of \eqref{e:COcommute} now follows from \cite[Theorem 1.5]{Al08}. 
Note the formalism in \cite{Al05, Al08} was corrected in \cite{Al10}, but these modifications affect neither our use of 
the closed-open map nor the commutativity of \eqref{e:COcommute}.

\subsubsection{Proving part (ii): The energy-capacity inequalities}

\begin{proof}[Proof of Theorem~\ref{t:CapIneq}(ii): Hamiltonian case]
If $e_{0} > e(X; M)$, then we can pick a Hamiltonian
$G$ with $\norm{G} < e_{0}$ such that $\vp_{G}^{1}$ displaces $X$.
Without loss of generality we can assume that $G$ is non-degenerate, $G \leq 0$,
and $\sup_{S^{1} \times M} \abs{G(t,x)} < e_{0}$.
Let $\cN$ be a neighborhood of $X$ so that $\vp_{G}^{1}$ displaces $\cN$ as well.

For any $\eps > 0$, pick a Hamiltonian $H > -\eps$ in $\cH^{X}$ that is non-degenerate
and equal to the constant $M_{H}$ outside $S^{1} \times \cN$ where $M_{H} > e_{0}$.  Pick $f \in \cH^{X}$ to be
a Hamiltonian adapted to $M$ such that $-\e < f \leq H$ and $M_{f} > e_{0}$.  Our choice of $f$
gives 
$$HF^{*}_{(-\e, 0]}(f) \cong HF^{*}_{(-\e - e_{0},0]}(f-e_{0}) \cong H^{*}(M)$$
via \eqref{e:CLIH}.  Refer to Figure~\ref{f:ECIGraph} for a schematic graph of these Hamiltonians.

\begin{figure}[h]
  \centering
  \def\svgwidth{425pt}
  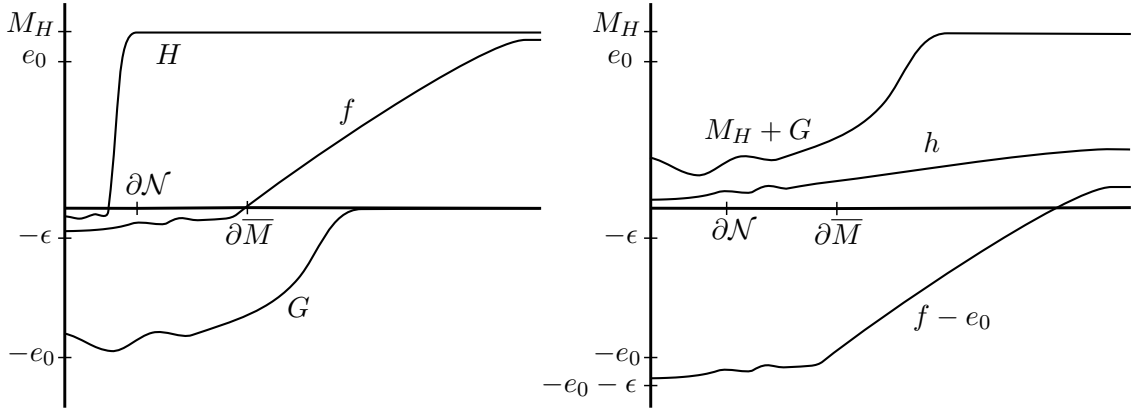
  \caption{The various Hamiltonians involved in the proof of Theorem~\ref{t:CapIneq}(ii).  Outside
  of $\overline{M}$ the Hamiltonian
  $f$ can be taken to be radial $f = f(r)$ on the convex end \eqref{e:end} with $\abs{f'(r)}$ smaller
  than the minimal period of a Reeb orbit on $\d\overline{M}$.}
  \label{f:ECIGraph}
\end{figure}

For $s \in [0,1]$ let $K^{s} = (1-s) H + s M_{H}$ and consider the family of admissible Hamiltonians
$$
(K^{s}\# G)_{t}: = K^{s}_{t} + G_{t}\circ(\vp_{K^{s}}^{t})^{-1}
$$
where $K^{s}\# G$ generates the Hamiltonian isotopy $\{\vp_{K^{s}}^{t}\vp_{G}^{t}\}_{t}$.
Since $\vp_{G}^{1}$ displaces $\cN$, while $\vp_{K^{s}}^{t}(\cN) = \cN$ and outside of 
$\cN$ one has $\vp_{K^{s}}^{t} = \id$, it follows that the fixed points of $\vp_{K^{s}}^{1}\vp_{G}^{1}$
and $\vp_{G}^{1}$ coincide.  In particular the fixed points of $\vp_{K^{s}}^{1}\vp_{G}^{1}$
are $s$-independent and under the correspondence between fixed points and
orbits in $\cC_{K^{s}\# G}$ it is known \cite[Chapter 5.5]{HZ94} that the
actions $\cA_{K^{s}\#G}$ are also $s$-independent.  It follows that the (non-monotone)
homotopy $K^{s}\#G$ induces an isomorphism
\begin{equation}\label{e:isoSpec}
	HF^{*}_{(-a, 0]}(H\# G) \stackrel{\cong}{\longrightarrow} HF^{*}_{(-a, 0]}(M_{H} + G)
\end{equation}
see for instance \cite[Section 3.2.3]{Gi07} or \cite{BPS03, FH94, Vi99}.
Since $f-e_{0} \leq K^{s}\#G$ for all $s$, 
the isomorphism \eqref{e:isoSpec}
actually fits into the commutative diagram
\begin{equation}\label{e:GCD}
	\xymatrix{
	HF^{*}_{(-e_{0} - \e, 0]}(f-e_{0}) \ar[dr]^{\Phi_{1}} \ar[d]_{\Phi_{0}}\\
	HF^{*}_{(-e_{0} - \e, 0]}(H\# G) \ar[r]^{\cong} & HF^{*}_{(-e_{0} - \e, 0]}(M_{H} + G)}
\end{equation}
where $\Phi_{0}$ and $\Phi_{1}$ are monotone continuation maps \cite[Section 2.2.2]{Gi10}.

Since $M_{H} + G > 0$, we can factor the monotone continuation map
$\Phi_{1}$ into two monotone continuation maps
\begin{equation}\label{e:Factor}
	HF^{*}_{(-e_{0} - \e, 0]}(f-e_{0}) \to HF^{*}_{(-e_{0}-\e, 0]}(h) \to HF^{*}_{(-e_{0}-\e, 0]}(M_{H} + G)
\end{equation}
where $h: M \to \R$ is an admissible Hamiltonian whose only $1$-periodic orbits are critical points
and is such that $0 < h \leq M_{H} + G$.
Since these conditions on $h$ imply that $HF^{*}_{(-e_{0}-\e, 0]}(h) = 0$, we know that
$\Phi_{1} = 0$ and therefore by \eqref{e:GCD} that $\Phi_{0} = 0$.  We 
also have the commutative diagram of monotone continuation maps
$$
	\xymatrix{
	HF^{*}_{(-e_{0} - \e, 0]}(f) \ar[d]^{\Phi_{fH}} & HF^{*}_{(-e_{0}-\e, 0]}(f - e_{0}) \ar[l]_{\cong} 
	\ar[d]^{\Phi_{0} = 0}\\
	HF^{*}_{(-e_{0} - \e, 0]}(H) & HF^{*}_{(-e_{0} - \e, 0]}(H \# G) \ar[l]
	}
$$
Since the top map is an isomorphism we have that the continuation map
$$\Phi_{fH}: HF^{*}_{(-e_{0} - \e, 0]}(f) \to HF^{*}_{(-e_{0} - \e, 0]}(H)$$
is zero. Therefore by Lemma~\ref{l:CapBoundHam}
we have $c^{FHW}(X) \leq e_{0} + \e$
and letting $e_{0}$ tend to $e(X; M)$ and $\e$ tend to $0$ gives the result.
\end{proof}

The proof of Theorem~\ref{t:CapIneq}(ii) in the Lagrangian case is analogous.  The only slight difference
is one takes a Hamiltonian $G$ that displaces $L$ from $X$ so that $G \leq 0$ and $\sup_{S^{1}\times L} \abs{G(t,x)} < e_{0}$.  Then at the part corresponding to \eqref{e:Factor}, one factors
\begin{equation}
	HF^{*}_{(-e_{0} - \e, 0]}(L;f-e_{0}) \to HF^{*}_{(-e_{0}-\e, 0]}(L;h) \to HF^{*}_{(-e_{0}-\e, 0]}(L;M_{H} + G)
\end{equation}
where $h: M \to \R$ is admissible, $h|_{L}$ is positive, and all chords $\cC_{h}(L)$ correspond to critical points of
$h|_{L}$.  This forces all chords 
$x \in \cC_{h}(L)$ to have positive action $\cA_{h,L}(x) > 0$ and hence $HF^{*}_{(-e_{0}-\e, 0]}(L;h) = 0$.



\nomenclature[1a]{$\cA_H$}{Hamiltonian action, see Equation~\eqref{e:HamiltonianActionFunctional}}
\nomenclature[1b]{$\cA_{H,L}$}{Lagrangian action, see Equation~\eqref{e:LagAct}}
\nomenclature[1c]{$\cA_{H,L}^{T^*Q}$}{Lagrangian cotangent bundle action, see Equation~\eqref{e:CTBA}}
\nomenclature[1d]{$C_H$}{Contractible chords, see Section~\ref{sss:chords}}
\nomenclature[1e]{$C_H^{*}$}{All chords, see Section~\ref{sss:chords}}
\nomenclature[1f]{$\cH$}{Admissible Hamiltonians, see Definition~\ref{d:cH}}
\nomenclature[1g]{$\cH^X$}{Admissible Hamiltonians for $X$, see Equation~\eqref{e:cHW}}
\nomenclature[1h]{$\cH^Q_{\cN,g}$}{Admissible Hamiltonians for $Q$ localized in $\cN$, see Definition~\ref{d:cHcN}}
\nomenclature[1i]{$f_H,\epsilon_H,\rho_H$}{Terms used to define elements in $\cH^Q_{\cN,g}$, see Definition~\ref{d:cHcN}}
\nomenclature[1j]{$\cJ_\th$}{Admissible almost complex structures, see Definition~\ref{d:cJ}}
\nomenclature[1k]{$\cJ_{\iota}$}{$J \in \cJ_\th$ that are standard on the image of $\iota$,
see Definition~\ref{d:JinBall}}
\nomenclature[1l]{$\cJ_{\cyl, g}(\cN)$}{$J \in \cJ_\th$ of contact type near $\partial\cN$, see Definition~\ref{d:CylJinN}}
\nomenclature[1m]{near/far chord}{Types of chords, see Section~\ref{s:cHcN}}
\nomenclature[1n]{path/loop chord}{Types of chords, see Section~\ref{s:cHcN}}
\nomenclature[1o]{$x_n$, $x_f$}{The near and far chords corresponding to a geodesic, see Section~\ref{s:cHcN}}

\printnomenclature[1.5in]

%



\bigskip

\noindent
\begin{center}
\begin{tabular}{ll}
Matthew Strom Borman & Mark McLean\\
Stanford University &  University of Aberdeen \\
borman@stanford.edu & mmclean@abdn.ac.uk
\end{tabular}
\end{center}


\end{document}